\documentclass[12pt]{amsart}

\usepackage{amsfonts,latexsym}
\usepackage{amsmath}                      
\usepackage{amsthm,amscd,amssymb} 
\usepackage{dsfont}
\usepackage{graphicx}
\usepackage{epsfig}
\usepackage[latin1]{inputenc}
\usepackage[T1]{fontenc}

\newtheorem{theorem}{Theorem}[section]
\newtheorem{lemma}[theorem]{Lemma}
\newtheorem{prop}[theorem]{Proposition}
\newtheorem{cor}[theorem]{Corollary}

\newtheorem{assumption}[theorem]{Assumption}

\theoremstyle{definition}

\newtheorem{remark}[theorem]{Remark}

\numberwithin{equation}{section}

\newcommand{\rr}{{\mathbb R}}
\newcommand{\rd}{{\mathbb R^d}}

\newcommand{\nat}{{\mathbb N}}
\newcommand{\ganz}{{\mathbb Z}}
\newcommand{\Exp}{{\mathbb E}}
\newcommand{\Ker}{\operatorname{Ker}}

\newcommand{\supp}{\operatorname{supp}}

\newcommand{\eqd}{\stackrel{\rm d}{=}}

\newcommand{\Prob}{\mathbb P}
\allowdisplaybreaks

\begin{document}

\sloppy
\title[Exact Hausdorff Measure functions]{On exact Hausdorff Measure Functions of operator semistable L\'evy processes} 

\author{Peter Kern}
\address{Peter Kern, Mathematical Institute, Heinrich-Heine-University D\"usseldorf, Universit\"atsstr. 1, D-40225 D\"usseldorf, Germany}
\email{kern\@@{}hhu.de} 

\author{Lina Wedrich}
\address{Lina Wedrich, Mathematical Institute, Heinrich-Heine-University D\"usseldorf, Universit\"atsstr. 1, D-40225 D\"usseldorf, Germany}
\email{lina.wedrich\@@{}uni-duesseldorf.de} 

\date{\today}

\begin{abstract}
Let $X=\{X(t)\}_{t\geq0}$ be an operator semistable L\'evy process on $\mathbb{R}^d$ with exponent $E$, where $E$ is an invertible linear operator on $\mathbb{R}^d$. In this paper we determine exact Hausdorff measure functions for the range of $X$ over the time interval $[0,1]$ under certain assumptions on the principal spectral component of $E$. As a byproduct we also present Tauberian results for semistable subordinators and sharp bounds for the asymptotic behavior of the expected sojourn times of $X$.
\end{abstract}

\thanks{This work was supported by Deutsche Forschungsgemeinschaft (DFG) under grant KE1741/6-1}

\maketitle

\baselineskip=18pt


\section{Introduction}

Let $X=\{X(t)\}_{t\geq0}$ be a L\'evy process on $\rd$. More precisely, $X$ is a stochastically continuous process with c\`adl\`ag paths and stationary and independent increments that starts in $X(0)=0$ almost surely. Then the distribution of $X$ is uniquely determined by the distribution of $X(1)$ which can be an arbitrary infinitely divisible distribution. For $c>1$ and a linear operator $E$ on $\rd$ we call the L\'evy process $X$ $(c^E,c)$-operator semistable if the distribution of $X(1)$ is full, i.e. not supported on any lower dimensional hyperplane, and
\begin{equation}\label{opsem}
\left\{X(ct)\right\}_{t\geq0}\stackrel{\text{fd}}{=} \left\{c^E X(t)\right\}_{t\geq0},
\end{equation}
where $\stackrel{\text{fd}}{=}$ denotes equality of all finite-dimensional distributions and
$$c^E:=\sum_{n=0}^\infty \frac{(\log c)^n}{n!}E^n.$$
The linear operator $E$ is referred to as the exponent of the operator semistable L\'evy process $X$. If \eqref{opsem} holds for all $c>0$, the L\'evy process is called operator stable. If the exponent $E$ is a multiple of the identity, i.e.\ $E=1/\alpha\cdot I$, where necessarily $\alpha\in(0,2]$, the process $X$ is simply called $(c^{1/\alpha},c)$-semistable. In case \eqref{opsem} holds for all $c>0$, the L\'evy process is called operator stable with exponent $E$, or $\alpha$-stable in case $E=1/\alpha\cdot I$, where $\alpha=2$ refers to the Brownian motion case.

In the past, efforts have been made to generate results on exact Hausdorff measure functions for the range of stable L\'evy processes. The case of Brownian motion was studied by Ciesielski and Taylor \cite{CT,T1}. An exact Hausdorff measure function for the range of an $\alpha$-stable L\'evy process was formulated by Taylor \cite{T}. It turned out that the gauge function depends on whether the continuous Lebesgue density of $X(1)$ is positive or zero in the origin, which by Taylor were called stable processes of type $A$ or type $B$, respectively. Furthermore, Pruitt and Taylor \cite{PT} studied sample path properties of L\'evy processes with independent stable components, including exact Hausdorff measures. Based on their work, Hou and Ying \cite{HY} determined an exact Hausdorff measure function for the range of certain operator stable L\'evy process of type $A$ with diagonal exponent $E$.  They emphasize without proof that similar methods also lead to an exact Hausdorff measure function for type $B$, see Remark 1 in \cite{HY}. Corresponding results for the graph are presented in \cite{Hou}. For an overview on general dimension results for L\'evy processes see \cite{KX} and \cite{Xiao}.

Our aim is to generalize the results of Hou and Ying in three respects. Firstly, we consider the more general class of operator semistable L\'evy process with the weaker discrete scaling \eqref{opsem}, secondly, we will relax the assumption that the exponent $E$ should be diagonal and show that it suffices to require diagonality for a principal spectral component; see Section 2.2 for details. Lastly, we will also derive exact Hausdorff measure functions for type B which turned out to be more challenging than asserted in Remark 1 of \cite{HY}. The Hausdorff dimension for the range and the graph of operator semistable L\'evy processes have recently been determined in \cite{KW} and \cite{W}, respectively; see also \cite{KMX} for an alternative derivation based on an index formula presented in \cite{KXZ}. The special case of the limit process in subsequent coin-tossing games of the famous St.\ Petersburg paradox has been studied in \cite{KW2} in detail.

The methods applied in this paper are similar to the ones used in \cite{T} and \cite{HY} with complementary work necessary to handle the weaker semistable scaling \eqref{opsem} or the non-diagonality of the operator $E$. The paper is structured as follows. Section 2.1 gives the definition of an exact Hausdorff measure function for an arbitrary Borel set $F\subseteq\rd$. In Section 2.2 we recall spectral decomposition results as stated in \cite{MS} which enable us to decompose the operator semistable L\'evy process $X$ according to the distinct real parts of the eigenvalues of the exponent $E$. Sharp bounds for the expected sojourn times of operator semistable L\'evy processes are presented in Section 2.3. In case $X$ is of type $B$ we will need an appropriate estimate of the behavior of the process near the origin following from certain Tauberian results presented in Section 2.4. The main results are stated and proven in Section 3 and 4, respectively. 

Throughout this paper, $K$ denotes an unspecified positive and finite constant that can vary in each occurrence, whereas fixed constants will be denoted by $K_1, \tilde K_1, C_1, K_2, \tilde K_2, C_2$, etc.

\section{Preliminaries}
\subsection{Exact Hausdorff measure functions}
A function $\phi$ is said to belong to the class $\Phi$ if there exists a constant $\delta>0$ such that $\phi$ is right-continuous and increasing on the open interval $(0,\delta)$, $\phi(0+)=0$ and fulfills the doubling property, i.e. there exists a constant $K_1>0$ such that
\begin{equation}\label{PropPhi}
\frac{\phi(2s)}{\phi(s)}\leq K_1 \quad \text{ for all } 0<s<\frac12\delta.
\end{equation} 
For a function $\phi\in\Phi$ the $\phi$-Hausdorff measure of an arbitrary Borel set $F\subseteq\rd$ is then defined as
\begin{equation}
\phi-m(F)=\liminf_{\epsilon\to\infty}\left\{\sum_{i=1}^\infty \phi(|F_i|):F\subseteq\bigcup_{i=1}^{\infty}F_i,|F_i|<\epsilon\right\},
\end{equation}
where $|F|=\sup\{\|x-y\|:x,y\in F\}$ denotes the diameter of a set $F\subseteq\rd$ and $\|\cdot\|$ is the Euclidean norm. The function $\phi\in\Phi$ is called an exact Hausdorff measure function for $F\subseteq\rd$ if $0<\phi-m(F)<\infty$. We refer to \cite{Fal} for a comprehensive introduction to exact Hausdorff measures. We emphasize that all the gauge functions $\phi$ appearing in this paper belong to the class $\Phi$.

For an arbitrary Borel measure $\mu$ on $\rd$ and a function $\phi\in\Phi$, the upper $\phi$-density of $\mu$ at $x\in\rd$ is defined as 
\begin{equation}
\overline{D}_\mu^\phi=\limsup_{r\to0}\frac{\mu(B(x,r))}{\phi(2r)},
\end{equation}
where $B(x,r)$ denotes the closed ball with radius $r$ centered at $x$. The following lemma is similar to Lemma 2.1 in \cite{HY} and is a direct consequence of the results in \cite{RT}.

\begin{lemma}\label{UpperPhi}
	For a given $\phi\in\Phi$, there exists a positive constant $K_2$ such that for any Borel measure $\mu$ on $\rd$ and every Borel set $F\subseteq\rd$, we have
	\begin{equation}
	\phi-m(F)\geq K_2\;\mu(F)\inf_{x\in F} \frac{1}{\overline{D} _\mu^\phi(x)}.
	\end{equation}
\end{lemma}

\subsection{Spectral decomposition}\label{subsecSpecDec}

Let $X$ be a $(c^E,c)$-operator semistable L\'evy process. Factor the minimal polynomial of $E$ into $q_1(x)\cdot\ldots\cdot q_p(x)$ where all roots of $q_i$ have real parts equal to $a_i$ and $a_i\neq a_j$ for $i\neq j$. Without loss of generality, one can additionally assume that $a_i<a_j$ for $i<j$. Note that $a_j\geq\frac12$ for all $j\in\{1,\ldots p\}$ by Theorem 7.1.10 in \cite{MS}.  Define $V_j=\Ker(q_j(E))$. According to Theorem 2.1.14 in \cite{MS} $V_1\oplus\cdots\oplus V_p$ is then a direct sum decomposition of $\rd$ into $E$ invariant subspaces. In an appropriate basis, $E$ is then block-diagonal and we may write $E=E_1\oplus\cdots\oplus E_p$ where $E_j:V_j\rightarrow V_j$ and every eigenvalue of $E_j$ has real part equal to $a_j$. Especially, every $V_j$ is an $E_j$-invariant subspace of dimension $d_j= \dim V_j$ and $d=d_1+\ldots+d_p$. Write $X(t)=X^{(1)}(t)+\ldots+X^{(p)}(t)$ with respect to this direct sum decomposition, where by Lemma 7.1.17 in \cite{MS}, $X^{(j)}=\{X^{(j)}(t)\}_{t\geq0}$ is a $(c^{E_j},c)$-operator semistable L\'evy process on $V_j$. We can now choose an inner product $\langle\cdot,\cdot\rangle$ on $\rd$ such that the $V_j, j\in\{1, \ldots,p\}$, are mutually orthogonal and throughout this paper we will let $\|x\|=\sqrt{\langle x, x\rangle}$  be the associated Euclidean norm. In particular we have for $t=c^r m >0$ that
\begin{align}
\|X(t)\|^2\eqd\|c^{rE}X(m)\|^2=\|c^{rE_1}X^{(1)}(m)\|^2+\ldots+\|c^{rE_p}X^{(p)}(m)\|^2,
\end{align}
with $r\in\mathbb{Z}$ and $m\in[1,c)$. 

Throughout this paper, we will denote by $\alpha_j=1/a_j$ the reciprocals of the real parts of the eigenvalues of the exponent $E$. We assume that the process $X$ has no Gaussian component in which case $0<\alpha_p<\ldots<\alpha_1<2$. Note that in this paper, we will only consider operator semistable L\'evy processes with diagonal principal exponent, i.e.\ $E_1=\alpha_1^{-1}\cdot I^{d_1}$, where $I^{d_1}$ denotes the identity operator on the $d_1$-dimensional subspace $V_1$. Since $V_1\cong\rr^{d_1}$ we may consider $X^{(1)}$ as an operator semistable L\'evy process on $\rr^{d_1}$ with diagonal exponent $E_1=\alpha_1^{-1}\cdot I^{d_1}$ and identity matrix $I^{d_1}\in\rr^{d_1\times d_1}$. Unless otherwise stated, throughout this paper there will be no restriction on all the other spectral components $j=2,\ldots,p$, i.e.\ $X^{(j)}$ is an arbitrary $(c^{E_j},c)$-operator semistable L\'evy process on $V_j\cong\rr^{d_j}$, where the real part of any eigenvalue of the exponent $E_j$ is equal to $a_j=\alpha_j^{-1}\in(0,2)$, but in general we do not assume that $E_j$ is diagonal for $j=2,\ldots,p$.

\subsection{Expected sojourn times}
For a L\'evy process $X=\{X(t):t\geq0\}$ let 
	$$
	T(a,s) = \int_0^s 1_{B(0,a)}(X(t))dt,
	$$
be the sojourn time up to time $s>0$ in the closed ball $B(0,a)$ with radius $a>0$ and centered at the origin. We now determine sharp upper and lower bounds for the expected sojourn times $\Exp[T(a,s)]$ of an operator semistable L\'evy process with diagonal exponent $E$. Although, in this paper we only need the result for $\alpha_1<d_1$, for completeness we also include the result for $\alpha_1>d_1$.

\begin{lemma}\label{sojourndiag}
	Let $X$ be a $(c^E,c)$-operator semistable L\'evy process on $\rd$ with diagonal principal exponent $E_1$.
	\begin{itemize}
		\item[(i)] If $\alpha_1<d_1$, there exist constants $K_4$, $K_5>0$ such that for all $0<a\leq1$ and $a^{\alpha_1}\leq s\leq1$,
		$$
		K_4 a^{\alpha_1}\leq\Exp[T(a,s)]\leq K_5 a^{\alpha_1}.
		$$
		\item[(ii)] If $d\geq2$ and $\alpha_1>d_1$ then $d_1=1$ and we further assume that $E_2$ is diagonal. Then there exist constants $K_6, K_7>0$ such that for all $a>0$ small enough, say $0<a\leq a_0$, and all $a^{\alpha_2}\leq s\leq1$,
		$$
		K_6 a^{\rho}\leq\Exp[T(a,s)]\leq K_7 a^{\rho},
		$$
		where $\rho=1+\alpha_2(1-1/\alpha_1)$.
	\end{itemize}
\end{lemma}

\begin{proof}
	The assertions can be proven by only slightly varying the proof of Theorem 2.6 in \cite{KW} and using the fact that for $E_j=\alpha_j^{-1}\cdot I^{d_j}$, where $I^{d_j}\in\rr^{d_j\times d_j}$ denotes the identity operator on $V_j$, we have $\|t^{E_j}\|=t^{1/\alpha_j}$ for all $t\geq0$.
\end{proof}

\subsection{Tauberian results for R-O variation}

Throughout this section let $X$ be a $(c^{1/\alpha},c)$-semistable subordinator for some $0<\alpha<1$, i.e.\ a semistable L\'evy process on $\rr$ with almost surely increasing sample paths. By Proposition 14.5 and Theorem 14.7(i) in \cite{Sato} it follows that $X$ has no drift part and $\int_{\{|x|\leq1\}}|x|\,\nu(dx)<\infty$ for the L\'evy measure $\nu$. Hence by Theorem 2.1 in \cite{Bro} the support of the L\'evy measure necessarily is contained in $\rr_+$. By Corollary 7.4.4 in \cite{MS} we have
\begin{equation}\label{Lmtail}
\nu(x,\infty)=x^{-\alpha}\theta(\log x)\quad\text{ for all }x>0,
\end{equation}
where $\theta$ is a strictly positive and $\log c$-periodic function defined on $\rr$. Moreover, using (7.77) in \cite{MS} we easily get
\begin{equation}\label{Lmperb}
c^{-\alpha}\theta(0)\leq\theta(y)\leq c^{\alpha}\theta(0)\quad\text{ for all }y\in\rr.
\end{equation}
Our aim is to prove the following consequence of a variant of de Bruijn's Tauberian theorem; see Theorem 4.12.9 in \cite{BGT}.
\begin{theorem}\label{deBruijn}
There exists a constant $\tilde K_5>0$ such that 
$$\Prob\left(X(r)\leq x\right)\leq\exp\left(-\tilde K_5 x^{-\frac{\alpha}{1-\alpha}}\right)\quad\text{ for all $r\in[1,c]$ and $x>0$.}$$
\end{theorem}
Note that if $X$ is a stable subordinator then the function $\theta$ is constant and \eqref{Lmtail} shows that the L\'evy measure has a regularly varying tail. In this case it is well known that Theorem \ref{deBruijn} can be deduced from Tauberian results as in Theorem 4 of \cite{BT}; see also Theorem 8.2.2 in \cite{BGT}. Theorem \ref{deBruijn} can also be derived using asymptotic tail formulas for stable densities in \cite{Sko}; for a proof in case of a stable subordinator we refer to Theorem 2.5.2 in \cite{Zol}. In our more general semistable setup the precise tail asymptotic around zero is not available and due to \eqref{Lmtail} the tail of the L\'evy measure is not necessarily regularly varying. Hence the classical results for stable subordinators cannot be applied. Nevertheless, due to \eqref{Lmtail} and \eqref{Lmperb} the variation of the tail is of regular order which is called R-O variation in \cite{MS}. We will now show that there are corresponding Tauberian results leading to a proof of Theorem \ref{deBruijn}.

Let $\rho$ be the Borel measure on $\rr_+$ given by $\rho(0,x)=\int_0^xy\,\nu(dy)$ for $x>0$ then using \eqref{Lmtail} and the periodicity of $\theta$ we get 
\begin{align*}
\rho(0,x] & =\sum_{n=0}^{\infty}\int_{c^{-(n+1)}x}^{c^{-n}x}y\,\nu(dy)\leq\sum_{n=0}^{\infty}c^{-n}x\,\nu(c^{-(n+1)}x,c^{-n}x]\\
& =\sum_{n=0}^{\infty}c^{-n}x\left[(c^{-(n+1)}x)^{-\alpha}\theta\big(\log(c^{-(n+1)}x)\big)-(c^{-n}x)^{-\alpha}\theta\big(\log(c^{-n}x)\big)\right]\\
& =x^{1-\alpha}\theta(\log x)\sum_{n=0}^{\infty}c^{-n}(c^{(n+1)\alpha}-c^{n\alpha})=x^{1-\alpha}\theta(\log x)\frac{c^\alpha-1}{1-c^{\alpha-1}}.
\end{align*}
Similarly, we get the lower bound $\rho(0,x]\geq x^{1-\alpha}\theta(\log x)c^{-1}\frac{c^\alpha-1}{1-c^{\alpha-1}}$ for all $x>0$. Together this yields
\begin{equation}\label{rhobounds}
\frac{c^\alpha-1}{c-c^{\alpha}}\leq\frac{\rho(0,x]}{x^{1-\alpha}\theta(\log x)}\leq\frac{c^\alpha-1}{1-c^{\alpha-1}}\quad\text{ for all }x>0.
\end{equation}
For $s>0$ let $\bar\rho(s)=\int_0^\infty e^{-sx}\,\rho(dx)$ be the Laplace transform of the measure $\rho$.
\begin{lemma}\label{LTrho}
There exist constants $0<\tilde K_2<\tilde K_1$ such that
$$\tilde K_2 s^{1-\alpha}\theta(\log s)\leq\bar\rho(1/s)\leq\tilde K_1 s^{1-\alpha}\theta(\log s)\quad\text{ for all }s>0.$$
\end{lemma}
\begin{remark}
In particular, Lemma \ref{LTrho} shows that $\bar\rho(s)$ exists for all $s>0$ and that $\rho$ is an unbounded measure, since $\bar\rho(s)\to\infty$ as $s\downarrow0$.
\end{remark}
\begin{proof}
Using \eqref{Lmtail}, the periodicity of $\theta$ and \eqref{rhobounds}, for any $s>0$ we get 
\begin{align*}
\bar\rho(1/s) & =\int_0^se^{-x/s}\,\rho(dx)+\sum_{n=0}^\infty\int_{c^ns}^{c^{n+1}s}e^{-x/s}\,\rho(dx)\\
& \leq\rho(0,s]+\sum_{n=0}^\infty\exp(-c^n)\,\rho(c^ns,c^{n+1}s]=\rho(0,s]+\sum_{n=0}^\infty\exp(-c^n)\int_{c^ns}^{c^{n+1}s}y\,\nu(dy)\\
& \leq\rho(0,s]+\sum_{n=0}^\infty\exp(-c^n)c^{n+1}s\,\nu(c^ns,c^{n+1}s]\\
& =\rho(0,s]+\sum_{n=0}^\infty\exp(-c^n)c^{n+1}s\left[(c^ns)^{-\alpha}\theta(\log(c^ns))-(c^{n+1}s)^{-\alpha}\theta(\log(c^{n+1}s))\right]\\
& =\rho(0,s]+s^{1-\alpha}\theta(\log s)c(1-c^{-\alpha})\sum_{n=0}^\infty\exp(-c^n)c^{n(1-\alpha)}\leq\tilde K_1 s^{1-\alpha}\theta(\log s)
\end{align*}
and the lower bound follows analogously.
\end{proof}
Denote by $\mu_r$ the infinitely divisible distribution of $X(r)$ with L\'evy measure $r\cdot\nu$. Since $\int_0^1x\,\nu(dx)<\infty$ and $X$ has no drift part, for any $s>0$ and $r>0$ we can write the Laplace transform of $\mu_r$ as
\begin{equation}\label{LTmu}
\bar\mu_r(s)=\exp\left(-r\int_0^\infty(1-e^{-sx})\,\nu(dx)\right)=\exp(-r\psi(s)),
\end{equation}
where
$$\psi(s)=\int_0^\infty(1-e^{-sx})\,\nu(dx)=\int_0^\infty\frac{1-e^{-sx}}{x}\,\rho(dx).$$
By dominated convergence we further get
\begin{equation}\label{psider}\begin{split}
\psi'(s) & =\lim_{h\to0}\frac{\psi(s+h)-\psi(s)}{h}=\lim_{h\to0}\int_0^\infty\frac{e^{-xs}-e^{-x(s+h)}}{xh}\,\rho(dx)\\
&=\lim_{h\to0}\int_0^\infty e^{-xs}\,\frac{1-e^{-xh}}{xh}\,\rho(dx)=\bar\rho(s).
\end{split}\end{equation}
\begin{lemma}\label{mubar}
There exist constants $0<\tilde K_4<\tilde K_3$ such that
$$\tilde K_4 s^{\alpha}\leq -\frac1r\,\log\bar\mu_r(s)\leq\tilde K_3 s^{\alpha}\quad\text{ for all $s>0$ and $r>0$.}$$
\end{lemma}
\begin{proof}
Using \eqref{LTmu}, \eqref{psider}, Lemma \ref{LTrho} and \eqref{Lmperb} we get
\begin{align*}
-\frac1r\,\log\bar\mu_r(s) & =\psi(s)=\int_0^s\psi'(t)\,dt=\int_0^s\bar\rho(t)\,dt\leq\tilde K_1\int_0^st^{\alpha-1}\theta(\log(1/t))\,dt\\
&\leq \tilde K_1 c^\alpha\theta(0)\int_0^st^{\alpha-1}\,dt=\frac{\tilde K_1 c^\alpha\theta(0)}{\alpha}\,s^{\alpha}=\tilde K_3 s^{\alpha}
\end{align*}
and the lower bound follows analogously.
\end{proof}
\begin{proof}[Proof of Theorem \ref{deBruijn}]
Let $t>0$ be arbitrary but fixed. Then for any $s>0$ we have
\begin{align*}
\bar\mu_r(s^{-1/\alpha}) & =\int_0^\infty\exp\left(-xs^{-1/\alpha}\right)\,\mu_r(dx)\\
& \geq\int_0^{(s/t)^{(1-\alpha)/\alpha}}\exp\left(-\frac{xs^{(\alpha-1)/\alpha}}{s}\right)\,\mu_r(dx)\\
&\geq\exp\left(-\frac1s\,t^{\frac{\alpha-1}{\alpha}}\right)\,\mu_r[0,(s/t)^{(1-\alpha)/\alpha}].
\end{align*}
Write $x=(s/t)^{(1-\alpha)/\alpha}$ then $\frac1s=\frac1t\,x^{-\alpha/(1-\alpha)}$ and together with Lemma \ref{mubar} we get
\begin{align*}
\Prob\left(X(r)\leq x\right) & =\mu_r[0,x]\leq\exp\left(\frac1s\,t^{\frac{\alpha-1}{\alpha}}\right)\bar\mu_r(s^{-1/\alpha})\\
& \leq\exp\left(\frac1s\,t^{\frac{\alpha-1}{\alpha}}\right)\exp\left(-r\tilde K_4(s^{-1/\alpha})^{\alpha}\right)\\
& =\exp\left(-\frac1s\left(r\tilde K_4-t^{\frac{\alpha-1}{\alpha}}\right)\right)\\
& =\exp\left(-x^{-\frac{\alpha}{1-\alpha}}\left(r\tilde K_4\frac1t-\frac1{t^\alpha}\right)\right).
\end{align*}
Now choose $t>0$ small enough such that $r\tilde K_4\frac1t-\frac1{t^\alpha}=\tilde K_5>0$ for all $r\in[1,c]$ then the assertion follows.
\end{proof}


\section{Main Result}

Let $\alpha_1$ and $d_1$ be as defined in Section \ref{subsecSpecDec} by means of the spectral decomposition. As in \cite{HY} we were only able to fully solve the question of exact Hausdorff measures for the range of operator semistable L\'evy processes in the case $\alpha_1<d_1$ but also give partial results for the case $\alpha_1>d_1$. We will consider operator semistable L\'evy processes of type $A$ and type $B$, simultaneously. If $\alpha_1<d_1$ and $X$ is of type $B$ we will need the following assumption on the tail asymptotic of sojourn times.

\begin{assumption}\label{assumptionB1}
	Let $X$ be a $(c^E,c)$-operator semistable L\'evy process of type $B$ on $\rd$ with diagonal  principal exponent $E_1$ and $0<\alpha_1<1$. We suppose that there exist constants $K_{8}, \lambda_0>0$ such that for all $\lambda\geq\lambda_0$ and $a>0$
	\begin{align*}
	\Prob\left(T(a,1)> \lambda a^{\alpha_1}\right)\leq \exp\left(-K_{8}\lambda^{\frac1{1-\alpha_1}}\right).
	\end{align*}
\end{assumption}

Note that if $X$ is an operator stable L\'evy process of type $B$ with $\alpha_1<d_1$ and diagonal exponent $E$, then the projection of $X^{(1)}$ onto any coordinate-axis is a stable subordinator and thus necessarily $\alpha_1<1$. In this case it is known that Assumption \ref{assumptionB1} holds true by Lemma 6 in \cite{T} or Lemma 5.2 in \cite{PT}. In our more general operator semistable case it is an open question whether we have the same tail asymptotics of the sojourn times. The following result provides a sufficient condition for Assumption \ref{assumptionB1} to hold true.
\begin{prop}\label{Taubersatz}
(i) Let $X$ be a $(c^E,c)$-operator semistable L\'evy process on $\rd$ with diagonal  principal exponent $E_1$ and $0<\alpha_1<1$. Suppose the existence of a vector $u\in V_1\setminus\{0\}$ such that the support of the L\'evy measure $\nu$ of $X$ is contained in the halfspace $\mathbb H:=\{x\in\rd:\langle x,u\rangle\geq0\}$. Then $X$ is of type $B$ and Assumption \ref{assumptionB1} holds.
\end{prop}
\begin{proof}
Without loss of generality, we can assume that $\|u\|=1$. Since $X$ is strictly operator-semistable and $0<\alpha_1<1$, similar to Proposition 14.5 and Theorem 14.7(i) in \cite{Sato} it follows that $X$ has no drift part and $\int_{\{\|x\|\leq1\}}\|x\|\,\nu(dx)<\infty$. Hence by Theorem 2.1 in \cite{Bro} the support of $X(1)$ is equal to the closure of $\bigcup_{k\in\nat}\supp(\nu^{\ast k})$ which is contained in the halfspace $\mathbb H$ by assumption. Since $X(1)$ has a continuous Lebesgue density as shown in section 2.3 of \cite{KW}, it follows that $X$ is of type $B$.
Let $\tilde X=\{\tilde X(t)\}_{t\geq0}$ be the projection of $X$ onto the direction $u$, i.e.
$$\tilde X(t):=\langle X(t),u\rangle=\langle X^{(1)}(t),u\rangle\quad,\,t\geq0,.$$
Then clearly $\tilde X$ defines a L\'evy process on $\rr$ and
$$\tilde X(ct)=\langle X^{(1)}(ct),u\rangle=\langle c^{E_1}X^{(1)}(t),u\rangle=c^{1/\alpha_1}\langle X^{(1)}(t),u\rangle=c^{1/\alpha_1}\tilde X(t)$$
shows that $\tilde X$ is strictly $(c^{1/\alpha_1},c)$-semistable. Its L\'evy measure $\tilde\nu$ is given by projection of $\nu$ and thus $\supp(\tilde\nu)\subseteq\rr_+$ and $0\in\supp(\tilde\nu)$ by the semistable scaling property $\tilde\nu(c^{-m/\alpha_1}dx)=c^m\cdot\tilde\nu(dx)$ for all $m\in\ganz$. Hence the L\'evy process $\tilde X$ is concentrated on $\rr_+$ by Theorem 24.10 in \cite{Sato}. Since $0<\alpha_1<1$, Proposition 14.5 in \cite{Sato} shows that $\int_0^1x\,\tilde\nu(dx)<\infty$. Together with Theorem 14.7(i) in \cite{Sato} this shows that all the conditions of Theorem 21.5 in \cite{Sato} are fulfilled and we conclude that $\tilde X$ is a strictly $(c^{1/\alpha_1},c)$-semistable subordinator. 

Similar to the proof of Lemma 6 in \cite{T}, due to the almost surely increasing sample paths of $\tilde X$ we get
\begin{align*}
\Prob\left(T(a,1)> \lambda a^{\alpha_1}\right) & =\Prob\left(\int_0^11_{B(0,a)}(X(t))\,dt>\lambda a^{\alpha_1}\right)\\
& \leq\Prob\left(\int_0^11_{[0,a]}(\tilde X(t))\,dt\geq\lambda a^{\alpha_1}\right)=\Prob\left(\tilde X(\lambda a^{\alpha_1})\leq a\right).
\end{align*}
Write $a^{\alpha_1}=c^{m(a)}r(a)$ with $m(a)\in\ganz$, $r(a)\in[1,c)$ and $\lambda r(a)=c^{m(\lambda,a)}r(\lambda,a)$ with $m(\lambda,a)\in\nat$ for $\lambda\geq c$ and  $r(\lambda,a)\in[1,c)$. Then we get
\begin{align*}
\Prob\left(T(a,1)> \lambda a^{\alpha_1}\right) & \leq\Prob\left(\tilde X(c^{m(\lambda,a)}r(\lambda,a)c^{m(a)})\leq a\right)\\
& =\Prob\left(\tilde X(r(\lambda,a))\leq c^{-m(\lambda,a)/\alpha_1}c^{-m(a)/\alpha_1}a\right).
\end{align*}
Theorem \ref{deBruijn} implies that
\begin{align*}
\Prob\left(T(a,1)> \lambda a^{\alpha_1}\right) & \leq\exp\left(-\tilde K_5\left(c^{-m(\lambda,a)/\alpha_1}c^{-m(a)/\alpha_1}a\right)^{-\frac{\alpha_1}{1-\alpha_1}}\right)\\
& =\exp\left(-\tilde K_5\left(c^{-m(\lambda,a)}r(a)\right)^{-\frac{1}{1-\alpha_1}}\right)\\
&=\exp\left(-\tilde K_5\left(\frac{\lambda}{r(\lambda,a)}\right)^{\frac{1}{1-\alpha_1}}\right)\leq
\exp\left(-K_{8}\lambda^{\frac1{1-\alpha_1}}\right),
\end{align*}
where $K_8=\tilde K_5 c^{-1/(1-\alpha_1)}$ showing that Assumption \ref{assumptionB1} is fulfilled.
\end{proof}
\begin{remark}
We conjecture that the converse also holds, i.e.\ that the conditions of  Assumption \ref{assumptionB1} already imply the support condition in Proposition \ref{Taubersatz} so that Assumption \ref{assumptionB1} is superfluous. Due to Theorem 2.1 in \cite{Bro} we will need to show that if $X$ is of type $B$ then the support of $X(1)$ is contained in the halfspace $\mathbb H$. This is obviously fulfilled in case $d=1$. In the operator stable case this follows from the fact that 
$\mathcal K:=\{x\in V_1: p(t,x)>0\text{ for some }t>0\}$
is an open convex cone, where $x\mapsto p(t,x)$ denotes the continuous Lebesgue density of $X(t)$. But the arguments leading to this fact as given in section 3 of \cite{T} fail in case of the weaker semistable scaling property \eqref{opsem}.
\end{remark}
The following theorem states the main result of this paper.
\begin{theorem}\label{mainresult}
	Let $X$ be a $(c^E,c)$-operator semistable L\'evy process on $\rd$ with diagonal principal exponent $E_1$.
	\begin{itemize}
		\item[(i)] If $X$ is of type $A$ and $0<\alpha_1<\min\{2,d_1\}$ then
		\begin{align*}
		\phi(a)=a^{\alpha_1}\log\log\frac1a
		\end{align*}
		is an exact Hausdorff measure function for almost all sample paths of $X$ over the interval $[0,1]$.
		\item[(ii)] If $X$ is of type $B$ and $0<\alpha_1<1$, then, given Assumption \ref{assumptionB1},
		\begin{align*}
		\phi(a)=a^{\alpha_1}\left(\log\log\frac1a\right)^{1-\alpha_1}
		\end{align*}
		is an exact Hausdorff measure function for almost all sample paths of $X$ over the interval $[0,1]$.
	\end{itemize}
\end{theorem}

\section{Proof}

To prove our main result we will show that the asserted $\phi$-Hausdorff measures of the range of $X$ are both, greater than zero and less than infinity.


\subsection{Greater than zero}

The following tail asymptotic of the sojourn times is true for any L\'evy process and will be used if $X$ is of type A. The proof can be found in Lemma 3.2 of Hou and Ying \cite{HY} and uses the Markov inequality.

\begin{lemma}\label{MarkovA}
	Let $X$ be a L\'evy process on $\rd$. Then for all $0<\delta<1$, $\lambda>0$ and $a>0$, we have that
	\begin{equation}\label{Levytail}
	\Prob\left(T(a,1)>\lambda\Exp[T(2a,1)]\right)\leq \frac1{1-\delta}\cdot\exp{(-\delta\lambda)}.
	\end{equation}
\end{lemma}
\begin{remark}
Note that \eqref{Levytail} is only meaningful if $\lambda>1$ and in this case the right-hand side takes it minimum at $\delta=1-1/\lambda$. Thus for $\lambda>1$ the inequality \eqref{Levytail} becomes strongest in the form
$$\Prob\left(T(a,1)>\lambda\Exp[T(2a,1)]\right)\leq\lambda\exp(1-\lambda).$$
Note further that for an operator semistable L\'evy process of type $B$ with diagonal principal exponent $E_1$ and $\alpha_1<1$ by Lemma 2.2(i) our Assumption \ref{assumptionB1} is stronger than \eqref{Levytail} for large values of $\lambda$.
\end{remark}

\begin{lemma}\label{A1}
	Let $X$ be a $(c^E,c)$-operator semistable L\'evy process on $\rd$ with diagonal principal exponent $E_1$.
	\begin{itemize}
		\item[(i)] If $X$ is of type $A$ and $0<\alpha_1<\min\{2,d_1\}$ then for
		\begin{align*}
		\phi(a)=a^{\alpha_1}\log\log\frac1a
		\end{align*}
	there exists a positive constant $K_{91}$ such that for all $t_0\in[0,1]$ we have almost surely
	\begin{align}\label{condi}
	\limsup_{a\to0}\;\frac1{\phi(a)}\cdot\int_0^11_{B(X(t_0),a)}(X(t))\;dt\leq K_{91}.
	\end{align}
	\item[(ii)] If $X$ is of type $B$ and $0<\alpha_1<1$ then, given Assumption \ref{assumptionB1}, for
	\begin{align*}
	\phi(a)=a^{\alpha_1}\left(\log\log\frac1a\right)^{1-\alpha_1}
	\end{align*}
	there exists a positive constant $K_{92}$ such that for all $t_0\in[0,1]$ we have almost surely
	\begin{align}\label{condii}
	\limsup_{a\to0}\;\frac1{\phi(a)}\cdot\int_0^11_{B(X(t_0),a)}(X(t))\;dt\leq K_{92}.
	\end{align}
	\end{itemize}
\end{lemma}
Note that \eqref{condii} differs from \eqref{condi} since the definition of $\phi$ varies from type $A$ to type $B$.
\begin{proof}
	Let $t_0\in[0,1]$. Define
	\begin{align*}
	Y(t)=\left\{\begin{array}{ll} X(t_0)-X(t_0-t), & \text{if } 0\leq t <t_0 \\
	\\
	X(t), & \text{if } t\geq t_0\end{array}\right..
	\end{align*}
	Using a change of variable by setting $u:=t_0-t$ and $v:=t-t_0$ we get
	\begin{align*}
	& \int_0^1 1_{B(X(t_0),a)}(X(t))dt\\
	= & \int_0^{t_0} 1_{B(X(t_0),a)}(X(t))dt + \int_{t_0}^1 1_{B(X(t_0),a)}(X(t))dt\\
	= & -\int_{t_0}^0 1_{B(X(t_0),a)}(X(t_0-u))du + \int_0^{1-t_0}1_{B(X(t_0),a)}(X(v+t_0))dv\\
	= & \int_0^{t_0} 1_{B(0,a)}(Y(u))du + \int_0^{1-t_0} 1_{B(0,a)}(X(v+t_0)-X(t_0))dv\\
	\leq & \int_0^{1} 1_{B(0,a)}(Y(u))du + \int_0^{1} 1_{B(0,a)}(X(v+t_0)-X(t_0))dv.
	\end{align*}
	Note that the processes $\{X(t)\}_{t\geq0}$, $\{Y(t)\}_{t\geq0}$ and $\{X(t+t_0)-X(t_0)\}_{t\geq0}$ have the same finite-dimensional distributions. Hence, it is sufficient to show that there exists a constant $K_9>0$ such that
	\begin{equation*}
	\Prob\left(\limsup_{a\to0}\frac{T(a,1)}{\phi(a)}<\frac{K_9}2\right)=1.
	\end{equation*}
	
	For $X$ of type $A$, $0<\alpha_1<\min\{2,d_1\}$ and $a>0$ small enough, we have by Lemma \ref{sojourndiag}(i) and Lemma \ref{MarkovA} that
	\begin{equation*}^{}
	\Prob\left(T(a,1)>\frac{K_4}{2^{\alpha_1}}\lambda a^{\alpha_1}\right)\leq \frac1{1-\delta}\cdot\exp{(-\delta\lambda)}
	\end{equation*}
	for all $\delta\in(0,1)$ and all $\lambda>0$. Now choose $\lambda=\frac2\delta\log\log\frac1a$. Then for $a>0$ small enough
	\begin{equation}\label{A1eqn1}
	\Prob\left(T(a,1)>\frac{2K_4}{\delta}a^{\alpha_1}\log\log\frac1a\right)\leq \frac1{1-\delta}\cdot\left(\log\frac1a\right)^{-2}
	\end{equation}
	For $n\in\nat$ define $a_n:=2^{-n}$ and $E_n:=\{T(a_n,1)>\frac{2K_4}{\delta}\cdot a_n^{\alpha_1}\log\log\frac1{a_n}\}$. By \eqref{A1eqn1} we get for sufficiently large $N\in\nat$
	\begin{align*}
	\sum_{n=N}^{\infty}\Prob(E_n)\leq \frac1{1-\delta}\;\sum_{n=N}^{\infty}\left(\log\frac1{a_n}\right)^{-2} =\frac{(\log2)^{-2}}{1-\delta}\;\sum_{n=N}^{\infty}\frac1{n^2}<\infty.
	\end{align*}
	Applying the Borel-Cantelli lemma, for almost all $\omega$ there exists an integer $N(\omega)$ such that the event $E_n$ does not occur for $n\geq N(\omega)$. For $a>0$ small enough, we can find  $n_0\geq N(\omega)$ such that $a_{n_0+1}\leq a\leq a_{n_0}$ which gives us
	\begin{align*}
	\frac{T(a)}{a^{\alpha_1}\log\log\frac1a}\leq\frac{T(a_{n_0})}{a_{n_0+1}^{\alpha_1}\log\log\frac1{a_{n_0}}}\leq \frac{2K_4 a_{n_0}^{\alpha_1}\log\log\frac1{a_{n_0}}}{\delta a_{n_0+1}^{\alpha_1}\log\log\frac1{a_{n_0}}}<\frac{K_4 2^{1+\alpha_1}}{\delta}.
	\end{align*}
	For $K_{91}:=K_4\; 2^{2+\alpha_1}/\delta$ this concludes the proof of part (i).
	
	Now let $X$ be of type $B$ and $0<\alpha_1<1$. By Assumption \ref{assumptionB1} there exist positive constants $K_{8}$, $\lambda_0>0$ such that for all $\lambda\geq\lambda_0$
	\begin{align*}
	\Prob\left(T(a,1)> \lambda a^{\alpha_1}\right)\leq \exp\left(-K_{8}\lambda^{\frac1{1-\alpha_1}}\right).
	\end{align*}
	Put $\lambda = \left(\frac2{K_{8}}\cdot\log\log\frac1a\right)^{1-\alpha_1}$. For all $a>0$ sufficiently small, such that $\lambda\geq\lambda_0$, we then get
	\begin{align*}
	&\Prob\left(T(a,1)> \left(\frac2{K_{8}}\cdot\log\log\frac1a\right)^{1-\alpha_1}\cdot a^{\alpha_1}\right)\leq \exp\left(-K_{8}\cdot\left(\frac2{K_{8}}\log\log\frac1a\right)\right)\\
	& = \exp\left(-2\log\log\frac1a\right) = \left(\log\frac1a\right)^{-2}.
	\end{align*}
	Let $a_n = 2^{-n}$ and $E_n=\left\{T(a_n,1)> \left(\frac2{K_{8}}\cdot\log\log\frac1{a_n}\right)^{1-\alpha_1}\cdot a_n^{\alpha_1}\right\}$. Then for $N\in\nat$ sufficiently large
	\begin{align*}
	\sum_{n=N}^{\infty}\Prob(E_n)\leq\sum_{n=N}^{\infty}\left(\log \frac1{a_n}\right)^{-2} = (\log2)^{-2} \sum_{n=N}^{\infty} \frac1{n^2}<\infty.
	\end{align*}
	By Borel Cantelli, for almost all $\omega$ there exists an integer $N(\omega)$ such that $E_n$ does not occur for $n\geq N(\omega)$. If $a_{n+1}\leq a<a_n$ and $n\geq N(\omega)$
	\begin{align*}
	\frac{T(a)}{a^{\alpha_1}\left(\log\log\frac1a\right)^{1-\alpha_1}}\leq \frac{T(a_n)}{a_{n+1}^{\alpha_1}\left(\log\log\frac1{a_n}\right)^{1-\alpha_1}}\leq 2\cdot K_{8}^{\alpha_1-1}
	\end{align*}
	Setting $K_{92}:=4\cdot\left(K_{8}\right)^{\alpha_1-1}$ concludes the proof.
\end{proof}

\begin{theorem}\label{greaterthanzero}
	Let $X$ be a $(c^E,c)$-operator semistable L\'evy process on $\rd$ with diagonal principal exponent $E_1$.
	\begin{itemize}
		\item[(i)] If $X$ is of type $A$ and $0<\alpha_1<\min\{2,d_1\}$ then for
		\begin{align*}
		\phi(a)=a^{\alpha_1}\log\log\frac1a
		\end{align*}
		we have $\phi-m(X([0,1]))>0$ almost surely.
		\item[(ii)] If $X$ is of type $B$ and $0<\alpha_1<1$ then, given Assumption \ref{assumptionB1}, for
		\begin{align*}
		\phi(a)=a^{\alpha_1}\left(\log\log\frac1a\right)^{1-\alpha_1}
		\end{align*}
		we have $\phi-m(X([0,1]))>0$ almost surely.
	\end{itemize}
\end{theorem}

\begin{proof}
	For all subsets $A\in\rd$ define the random Borel measure $\mu$ as
	\begin{equation*}
	\mu(A)=\int_0^1 1_A(X(t))dt.
	\end{equation*}
	This gives us $\mu(X([0,1]))=1$ for all $\omega\in\Omega$. Let 
	$$F=\{X(t_0):t_0\in[0,1]\text{ and \eqref{condi}, resp.\ \eqref{condii} holds} \}\subseteq X([0,1]).$$ 
	By Tonelli's theorem we have almost surely
	\begin{align*}
	\mu(F) & = \int_0^1 1_F(X(t))dt = \int_0^1 1_{\{X(t_0)\;:\;t_0\in[0,1]\text{ and \eqref{condi}, resp.\ \eqref{condii} holds} \}}(X(t))dt\\
	& = \int_0^1\int_0^1 1_{\{X(t_0)\;:\text{ \eqref{condi}, resp.\ \eqref{condii} holds} \}}(X(t))dt_0\;dt\\
	& = \int_0^1\int_0^1 1_{\{X(t_0)\;:\text{ \eqref{condi}, resp.\ \eqref{condii} holds} \}}(X(t))dt\;dt_0 = \int_0^1 1 dt_0 = 1.
	\end{align*}
	Applying Lemma \ref{UpperPhi} and Lemma \ref{A1} and using the fact that $\phi$ is ultimately increasing, we have that almost surely
	\begin{align*}
	\phi-m(F)& \geq K_2 \;\mu(F)\inf_{X(t_0)\in F}\left(\limsup_{a\to\infty}\frac{\mu(B(X(t_0),a))}{\phi(2a)}\right)^{-1} \\
	& \geq K_2\cdot 1\cdot \inf_{X(t_0)\in F}\left(\limsup_{a\to\infty}\frac{\mu(B(X(t_0),a))}{\phi(2a)}\right)^{-1} \geq \frac{K_2}{\max\{K_{91},K_{92}\}}>0.
	\end{align*}
	Since $F\subseteq X([0,1])$ this concludes the proof.
\end{proof}

\begin{remark}
Similarly, if $X$ is of type $A$, $d\geq2$ and $\alpha_1>d_1=1$ then for 
\begin{equation}\label{phitypb}
	\phi(a)=a^{\rho}\log\log\frac1a
\end{equation}
with $\rho=1+\alpha_2(1-1/\alpha_1)$ we have $\phi-m(X([0,1]))>0$ almost surely. This follows analogously to the proof of Theorem \ref{greaterthanzero}(i) using Lemma \ref{sojourndiag}(ii) instead of part (i) in the proof of Lemma \ref{A1}. Unfortunately, in this case we were not able to show that $\phi-m(X([0,1]))<\infty$.
\end{remark}
 
 
 \subsection{Less than infinity}
 
 \begin{lemma}\label{lemmaA2}
 	Let $X$ be a $(c^E,c)$-operator semistable L\'evy process on $\rd$ with diagonal principal exponent $E_1$. 
 	\begin{itemize}
 		\item[(i)] If $X$ is of type $A$, then there exists a constant $K_{10}>0$ such that for all $0<\lambda<1$ and $0<\tau<1$
 		\begin{equation}
			\Prob\left(\sup_{0\leq t\leq \tau}\|X(t)\|\leq \tau^{\frac{1}{\alpha_1}}\lambda\right)\geq \exp\left(-K_{10}\lambda^{-\alpha_1}\right).
 		\end{equation}
 		\item[(ii)] If $X$ is of type $B$ and $0<\alpha_1<1$, then there exist constants $K_{11}, \lambda_0>0$ such that for all $0<\lambda<\lambda_0$ and $0<\tau<1$
 		$$
 		\Prob\left(\sup_{0\leq t\leq\tau}\|X(t)\|\leq\tau^{\frac1{\alpha_1}}\lambda\right)\geq\exp\left(-K_{11}\lambda^{-\frac{\alpha_1}{1-\alpha_1}}\right).
 		$$
 	\end{itemize}
 \end{lemma}
 
 \begin{proof}
 	(i) Let $p(t,.)$ be the density function of $X(t)$ for $t>0$. Since the process is of type $A$, the density function $p(1,\cdot)$ is bounded and continuous and $p(1,0)>0$. Hence, we can find $\delta, \eta>0$ such that for all $x\in\rd$ with $\|x\|<2\delta$ we have that $p(1,x)\geq\eta$.  
	Then for $\|x\|<\delta$ this leads to
 	\begin{align*}
 	&\Prob(\|X(1)+x\|<\delta) = \int_{\rd} 1_{\{\|y+x\|<\delta\}}\;p(1,y)dy\\
 	&\geq \int_{-\infty}^{\infty}\cdots\int_{-\infty}^{\infty} 1_{\{|y_1+x_1|<\frac\delta{\sqrt{d}}\}}\cdot\ldots\cdot1_{\{|y_d+x_d|<\frac\delta{\sqrt{d}}\}}\; p(1,y)\; dy_1\cdots dy_d\\
 	&= \int_{-2\delta}^{2\delta}\cdots\int_{-2\delta}^{2\delta} 1_{\{|y_1+x_1|<\frac\delta{\sqrt{d}}\}}\cdot\ldots\cdot1_{\{|y_d+x_d|<\frac\delta{\sqrt{d}}\}}\; p(1,y)\; dy_1\cdots dy_d\\
 	&\geq \;\eta\int_{-2\delta}^{2\delta}\cdots\int_{-2\delta}^{2\delta} 1_{\{|y_1+x_1|<\frac\delta{\sqrt{d}}\}}\cdot\ldots\cdot1_{\{|y_d+x_d|<\frac\delta{\sqrt{d}}\}}\; dy_1\cdots dy_d\\
 	&=\; \eta \prod_{i=1}^{d}\left(\frac\delta{\sqrt{d}}-x_i-\left(-\frac\delta{\sqrt{d}}-x_i\right)\right) = \eta \prod_{i=1}^{d}\left(\frac{2\delta}{\sqrt{d}}\right)=\eta \left(\frac{2\delta}{\sqrt{d}}\right)^d=:C_2>0.
 	\end{align*}
 	Furthermore, since the process $X$ has c\`adl\`ag paths, it is almost surely bounded on finite intervals. Hence, by tightness we can find $r>1$ large enough such that
 	\begin{equation*}
 	\Prob\left(\sup_{0\leq t\leq1}\|X(t)\|\geq r-\delta\right)<\frac12 C_2
 	\end{equation*}
 	Altogether, we get for all $\|x\|<\delta$
 	\begin{align*}
 	&\Prob\left(\sup_{0\leq t\leq1}\|X(t)+x\|<r,\|X(1)+x\|<\delta\right)\\
 	& =  \Prob\left(\|X(1)+x\|<\delta\right) - \Prob\left(\sup_{0\leq t\leq1}\|X(t)+x\|\geq r, \|X(1)+x\|<\delta\right)\\
 	& \geq  C_2 - \Prob\left(\sup_{0\leq t\leq1}\|X(t)+x\|\geq r\right)\\
 	& \geq C_2 - \Prob\left(\sup_{0\leq t\leq1}\|X(t)\|\geq r-\delta\right) > \frac12 C_2
 	\end{align*} 
 	Let $k\in\nat$. By induction, it now follows from the properties of a L\'evy process that
 	\begin{align*}
 	& \Prob\left(\sup_{0\leq t\leq k}\|X(t)\|<r\right)\geq \Prob\left(\sup_{0\leq t\leq k}\|X(t)\|<r, \|X(k)\|<\delta\right)\\
 	& \geq \Prob\left(\sup_{k-1\leq t\leq k}\|X(t)\|<r, \|X(k)\|<\delta, \sup_{0\leq t\leq k-1}\|X(t)\|<r, \|X(k-1)\|<\delta\right)\\
 	& \quad\quad\left. \sup_{0\leq t\leq k-1}\|X(t)\|<r, \|X(k-1)\|<\delta\right)\\
 	& = \int_{[0,r)}\int_{\{\|x\|<\delta\}} \Prob\left(\sup_{k-1\leq t\leq k}\|X(t)-X(k-1)+x\|<r, \|X(k)-X(k-1)+x\|<\delta\right)\\
 	&\hspace{3cm}\text{d}\Prob_{\left(\sup\limits_{0\leq t\leq k-1}\|X(t)\|, \|X(k-1)\|\right)}(x,y)\\
 	& = \int_{[0,r)}\int_{\{\|x\|<\delta\}} \Prob\left(\sup_{k-1\leq t\leq k}\|X(t)+x\|<r, \|X(1)+x\|<\delta\right)\\
 	&\hspace{3cm}\text{d}\Prob_{\left(\sup\limits_{0\leq t\leq k-1}\|X(t)\|, \|X(k-1)\|\right)}(x,y)\\
 	&\geq \frac12 C_2 \cdot \Prob\left(\sup_{0\leq t\leq k-1}\|X(t)\|<r, \|X(k-1)\|<\delta\right)\geq\Big(\frac12 C_2\Big)^k=\exp\left(-k\log\left(2C_2^{-1}\right)\right).
 	\end{align*}
 	For $u>1$ choose $k\in\nat$ with $k\leq u<k+1$. Then for all $r>1$ large enough we have
 	\begin{align*}
 	&\Prob\left(\sup_{0\leq t\leq u}\|X(t)\|<r\right) \geq \Prob\left(\sup_{0\leq t\leq k+1
 		}\|X(t)\|<r\right)\\
 		&\geq \exp\left(-(k+1)\log\left(2\;C_2^{-1}\right)\right) = \exp\left(-k\cdot\frac{k+1}{k}\cdot\log\left(2\;C_2^{-1}\right)\right)\\
 		&\geq \exp\left(-u\cdot 2\log\left(2\;C_2^{-1}\right)\right) =:\exp\left(-C_3\;u\right),
 	\end{align*}
 	where $C_3>0$ is a constant independent from $u$.
 	Now let $0<\tau<1$. Then there exists an $i\in\nat_0$ such that $c^{-(i+1)}\leq\tau<c^{-i}$ and for $0<\lambda<1$ there exists a $j\in\nat$ such that $c^{j-2}\leq\lambda^{-\alpha_1} r^{\alpha_1}<c^{j-1}$. Using the fact that for diagonal $E_1$ we have $\|s^E\|\leq s^{1/\alpha_1}$ for $0<s<1$, this leads us to
 	\begin{align*}
 	&\Prob\left(\sup_{0\leq t\leq \tau}\|X(t)\|\leq \tau^{\frac{1}{\alpha_1}}\lambda\right) \geq \Prob\left(\sup_{0\leq t\leq c^{-i}}\|X(t)\|\leq c^{-\frac{i+1}{\alpha_1}}\lambda\right)\\
 	& = \Prob\left(\sup_{0\leq t\leq c^{j}}\|X(c^{-j-i}t)\|\leq c^{-\frac{i+1}{\alpha_1}}\lambda\right)\geq \Prob\left(\sup_{0\leq t\leq c^{j}} c^{-\frac{j+i}{\alpha_1}} \|X(t)\|\leq c^{-\frac{i+1}{\alpha_1}}\lambda \right) \\
 	&\geq \Prob\left(\sup_{0\leq t\leq c^{j}}\|X(t)\|\leq c^{\frac{j-1}{\alpha_1}}\lambda\right) \geq \Prob\left(\sup_{0\leq t\leq c^{j}}\|X(t)\| < r\right) \\
 	&\geq \exp\left(-C_3 \cdot c^j\right) \geq \exp\left(-K_{10}\lambda^{-\alpha_1}\right),
 	\end{align*}
 	where $K_{10}:=C_3\cdot c^2\cdot r^{\alpha_1}$.
 	
 	(ii) Let $0<\tau<1$. Then there exists an $i_1\in\nat_0$ with $c^{-(i_1+1)}\leq\tau<c^{-i_1}$. We have
 	\begin{align*}
 	&\Prob\left(\sup_{0\leq t\leq\tau}\|X(t)\|\leq\tau^{\frac1{\alpha_1}}\lambda\right)\geq \Prob\left(\sup_{0\leq t\leq c^{-i_1}}\|X(t)\|\leq c^{-\frac{i_1+1}{\alpha_1}}\lambda\right)\\
 	&=\Prob\left(\sup_{0\leq t\leq 1}\|X(c^{-i_1}t)\|\leq c^{-\frac{i_1+1}{\alpha_1}}\lambda\right)\geq \Prob\left(\sup_{0\leq t\leq 1}\|X(t)\|\leq c^{-\frac{1}{\alpha_1}}\lambda\right)=:g(\lambda),
 	\end{align*}
 	where the last inequality follows from the fact that for diagonal $E_1$ we have $\|m\|^E\leq m^{1/\alpha_1}$ for all $m\in(0,1]$. Furthermore for all $k\in\nat\setminus\{1\}$ we can find a $j\in\nat_0$ such that $c^{-(j+1)}\leq k^{-1}<c^{-j}$. Since $X$ has independent and stationary increments we get
 	\begin{align*}
 	g(\lambda)&\geq\Prob\left(\bigcap_{i=1}^k\left\{\sup_{\frac{i-1}{k}\leq t\leq\frac i k}\left\|X(t)-X\left(\frac{i-1}{k}\right)\right\|<k^{-1}c^{-\frac1{\alpha_1}}\lambda\right\}\right)\\
 	&=\prod_{i-1}^k\Prob\left(\sup_{0\leq t\leq k^{-1}}\|X(t)\|<k^{-1}c^{-\frac1{\alpha_1}}\lambda\right) \\
 	&\geq \left[\Prob\left(\sup_{0\leq t\leq c^{-j}}\|X(t)\|<c^{-(j+1)}c^{-\frac1{\alpha_1}}\lambda\right)\right]^k\\
 	&= \left[\Prob\left(\sup_{0\leq t\leq 1}\|X(c^{-j}t)\|<c^{-j}c^{-\frac1{\alpha_1}-1}\lambda\right)\right]^k\\
 	&\geq \left[\Prob\left(\sup_{0\leq t\leq 1}\|X(t)\|<c^{j(\frac1{\alpha_1}-1)}c^{-\frac1{\alpha_1}-1}\lambda\right)\right]^k\\
 	&= \left[\Prob\left(\sup_{0\leq t\leq 1}\|X(t)\|<c^{(j+1)(\frac1{\alpha_1}-1)}c^{-\frac2{\alpha_1}}\lambda\right)\right]^k\\
 	&\geq \left[\Prob\left(\sup_{0\leq t\leq 1}\|X(t)\|<k^{\frac1{\alpha_1}-1}c^{-\frac2{\alpha_1}}\lambda\right)\right]^k =\left[g\left(k^{\frac1{\alpha_1}-1}c^{-\frac1{\alpha_1}}\lambda\right)\right]^k
 	\end{align*}
 	Define $h(\lambda)=\log g(\lambda)$. Then
 	$$
 	h(\lambda)\geq k\cdot h\left(k^{\frac1{\alpha_1}-1}c^{-\frac1{\alpha_1}}\lambda\right).
 	$$
 	Furthermore define the sequence $(x_k)_{k\in\nat}$ as $x_k=k^{1-\frac1{\alpha_1}}c^{\frac1{\alpha_1}}$, then $x_{k+1}/x_k\rightarrow1$ and $x_k\rightarrow0$ as $k\rightarrow\infty$. We get
 	\begin{align*}
 	h(x_k)\geq k\cdot h(1)=c^{\frac1{1-\alpha_1}}\cdot x_k^{\frac{\alpha_1}{\alpha_1-1}}\cdot h(1)\geq x_k^{-\frac{\alpha_1}{1-\alpha_1}}\cdot h(1).
 	\end{align*}
  Since for all $k\in\nat$ we have
 	$$
  g(1)\geq\left[\Prob\left(\sup_{0\leq t\leq 1}\|X(t)\|<k^{\frac1{\alpha_1}-1}c^{-\frac2{\alpha_1}}\right)\right]^k,
 	$$
 	there exists a $k_0\in\nat$ such that the right-hand-side is strictly positive for all $k\geq k_0$, i.e.\ $g(1)>0$. Hence, there exists a finite constant $K>0$ such that $h(1)=\log g(1)\geq -K$. Since $h$ is non-increasing, there is a $\lambda_0>0$ such that for $0<\lambda\leq \lambda_0$ with $x_{k+1}\leq\lambda<x_k$ and $k\geq k_0$ we have
 	\begin{align*}
 	h(\lambda)&\geq h(x_{k+1})=(x_{k+1})^{-\frac{\alpha_1}{1-\alpha_1}}\cdot h(1)\geq -K\cdot (x_{k+1})^{-\frac{\alpha_1}{1-\alpha_1}}\\
 	&=-K\cdot \left(\frac{x_{k+1}}{x_k}\right)^{-\frac{\alpha_1}{1-\alpha_1}}\cdot(x_{k})^{-\frac{\alpha_1}{1-\alpha_1}}\geq -K_{11}\;\lambda^{-\frac{\alpha_1}{1-\alpha_1}}.
 	\end{align*}
	Altogether we arrive at
 	$$
 	g(\lambda)\geq\exp\left(-K_{11}\;\lambda^{-\frac{\alpha_1}{1-\alpha_1}}\right)
 	$$ 	
 	for all $0<\lambda<\lambda_0$.
 \end{proof}
 
 \begin{lemma}\label{TailBehav}
Let $X$ be a $(c^E,c)$-operator semistable L\'evy process on $\rr^{d}$ with diagonal principal exponent $E_1$.
\begin{itemize}
	\item[(i)] For the principal component $j=1$ there exists a constant $K_{12}>0$ such that for all $i\in\mathbb{Z}$ and all $a>0$ we have
 	\begin{equation}\label{x1tail}
 	\Prob\left(\|X^{(1)}(c^{-i})\|>a c^{-\frac{i}{\alpha_1}}\right)=\Prob(\|X^{(1)}(1)\|>a) \leq K_{12}\;a^{-\alpha_1}.
 	\end{equation}
	\item[(ii)] For all other components $j=2,\ldots,p$ and arbitrary $\delta'>0$, $\delta_j\in(0,\alpha_j^{-1})$  there exists a constant $K_{j2}>0$ such that for all $i\in\mathbb{Z}$ and all $a\geq a_0\geq1$ we have
 	\begin{equation}\label{xjtail}
 	\Prob\left(\|X^{(j)}(c^{-i})\|>a c^{-i(\frac{1}{\alpha_j}-\delta_j)}\right)\leq K_{j2}\;a^{-(\alpha_j-\delta')}.
 	\end{equation}
\end{itemize}
 \end{lemma}
 
 \begin{proof}
 	(i) Let $\nu$ be a $(c^{1/\alpha_1},c)$-semistable law on $\rr$. One can show (see the Remarks in section 3 of \cite{MS2}) that for all $t>0$
 	\begin{equation}\label{SemistableLaw}
 	\nu(\{|x|>t\})=t^{-\alpha}f(t),
 	\end{equation}
 	where $f$ is a bounded, asymptotically log-periodic function. Let $X_1,\ldots,X_{d_1}$ denote the marginals 	of $X^{(1)}(1)$, i.e.\ $X_j=\langle X^{(1)}(1),e_j\rangle$ with canonical basis vector $e_j$, and let $X_j(t)=\langle X^{(1)}(t),e_j\rangle$ be the L\'evy process generated by $X_j$. Since
	$$X_j(ct)=\langle X^{(1)}(ct),e_j\rangle=\langle c^{E_1}X^{(1)}(t),e_j\rangle=c^{1/\alpha_1}\langle X^{(1)}(t),e_j\rangle=c^{1/\alpha_1}X_j(t)$$
	the distributions of the marginals $X_j$ are $(c^{1/\alpha_1},c)$-semistable on $\rr$. Hence, by \eqref{SemistableLaw} there exists a finite constant $C_1>0$ such that
 	\begin{align*}
 	\Prob\left(|X_j|>a\right)\leq C_1a^{-\alpha_1} \quad \text{ for all } \quad a>0,\; j = 1,\ldots,d_1.
 	\end{align*}
 	Since $\|X^{(1)}(1)\|\leq K\cdot \|X^{(1)}(1)\|_1 = K\cdot\sum_{j=1}^{d_1}|X_j|$,
 	we further get
 	\begin{align*}
 		\Prob\left(\|X^{(1)}(1)\|>a\right) & \leq \Prob\left(\sum_{j=1}^{d_1}|X_j|>\frac{a}{K}\right)\leq\Prob\left(\bigcup_{j=1}^{d_1}\left\{|X_j|>\frac{a}{d_1 K}\right\}\right)\\
 		&\leq \sum_{j=1}^{d_1}\Prob\left(|X_j|>\frac{a}{d_1 K}\right)\leq C_1\left(\frac{a}{d_1 K}\right)^{-\alpha_1} =: K_{12}\;a^{-\alpha_1},
 	\end{align*}
 	which concludes the proof of (i).
	
(ii) By Lemma 2.1 in \cite{KW} we have for any $r\in[1,c)$
\begin{equation}\begin{split}\label{xjdec}
\|X^{(j)}(rc^{-i})\| & =\|c^{-iE_j}X^{(j)}(r)\|\leq\|c^{-iE_j}\|\|X^{(j)}(r)\|\\
& \leq K\cdot c^{-i(a_j-\delta_j)}\|X^{(j)}(r)\|=K\cdot c^{-i(\frac1{\alpha_j}-\delta_j)}\|X^{(j)}(r)\|
\end{split}\end{equation}
and hence we get
\begin{equation}\label{xjtail1}
\Prob\left(\|X^{(j)}(c^{-i})\|>a c^{-i(\frac{1}{\alpha_j}-\delta_j)}\right)\leq\Prob\left(\|X^{(j)}(1)\|>K^{-1}a\right).
\end{equation}
As in part (i), for the marginals $X_1,\ldots,X_{d_j}$ of $X^{(j)}(1)$, i.e.\ $X_k=\langle X^{(j)}(1),e_k\rangle$ with canonical basis vector $e_k$, we get
\begin{equation}\label{xjtail2}
\Prob\left(\|X^{(j)}(1)\|>K^{-1}a\right)\leq \sum_{k=1}^{d_j}\Prob\left(|X_k|>C_2 a\right).
\end{equation}
In view of Theorem 8.2.1 in \cite{MS}, an application of Theorem 6.3.25(a) in \cite{MS} gives
\begin{equation}\label{xjtail3}
\Prob\left(|X_k|>C_2 a\right)=\Prob\left(|\langle X^{(j)}(1),e_k\rangle|>C_2 a\right)\leq C_{2k}a^{-\frac1{a_j}+\delta'}=C_{2k}a^{-(\alpha_j-\delta')}
\end{equation}
for all $a\geq a_0$ and some $a_0\geq 1$ independent of $k=1,\ldots,d_j$. Now, \eqref{xjtail} follows directly from \eqref{xjtail1}--\eqref{xjtail3}.
\end{proof}
 
 \begin{lemma}\label{LemmaA2A}
Let $X$ be a $(c^E,c)$-operator semistable L\'evy process on $\rr^{d}$ with diagonal principal exponent $E_1$. Given $\varepsilon\in(0,1)$, $\delta_1:=0$ and $\delta_j\in(0,\alpha_j^{-1})$ for $j=2,\ldots,p$, there exists a constant $a_0>0$ such that for all $a\geq a_0$ and all $i\in\nat_0$ we have
 	$$
 	\sup_{t\in[0,c^{-i}]}\Prob\left(\|X^{(j)}(t)\|>ac^{-i(\frac{1}{\alpha_j}-\delta_j)}\right)\leq\varepsilon<1.
 	$$
 \end{lemma}
 
 \begin{proof}
  Using \eqref{xjdec} in case $j=2,\ldots,p$ and the semistability in case $j=1$, we get
 	\begin{align*}
 	&\sup_{t\in[0,c^{-i}]}\Prob\left(\|X^{(j)}(t)\|>ac^{-i(\frac{1}{\alpha_j}-\delta_j)}\right)
	\leq\sup_{r\in[1,c)}\sup_{k\geq i} \;\Prob\left(\|X^{(j)}(rc^{-k})\|>ac^{-i(\frac{1}{\alpha_j}-\delta_j)}\right)\\
 	&\quad\leq\sup_{r\in[1,c)}\sup_{k\geq i} \;\Prob\left(K\cdot c^{-k(\frac{1}{\alpha_j}-\delta_j)}\|X^{(j)}(r)\|>ac^{-i(\frac{1}{\alpha_j}-\delta_j)}\right)\\
 	&\quad= \sup_{r\in[1,c)}\;\Prob\left(\|X^{(j)}(r)\|>\frac{a}{K}\right) .
 	\end{align*}
 	Since $\left(X^{(j)}(r)\right)_{r\in[1,c)}$ is stochastically continuous and hence weakly relatively compact, it follows by Prohorov's theorem that for $\varepsilon\in(0,1)$ there exists $a_0>0$ such that for all $a\geq a_0$ we have
 	$$
 	\sup_{r\in[1,c)}\;\Prob\left(\|X^{(j)}(r)\|>\frac{a}{K}\right)\leq \varepsilon <1,
 	$$
  concluding the proof.
 \end{proof}
 
 \begin{lemma}\label{LemmaA2}
Let $X$ be a $(c^E,c)$-operator semistable L\'evy process on $\rr^{d}$ with diagonal principal exponent $E_1$. Given $\varepsilon\in(0,1)$, $\delta_1:=0$ and $\delta_j\in(0,\alpha_j^{-1})$ for $j=2,\ldots,p$, there exists a constant $a_0>0$ such that for all $a\geq a_0$ and all $i\in\nat_0$ we have
 	$$
 	\Prob\left(\sup_{t\in[0,c^{-i}]}\|X^{(j)}(t)\|>2ac^{-i(\frac{1}{\alpha_j}-\delta_j)}\right)\leq\frac{1}{1-\varepsilon}\cdot\Prob\left(\|X^{(j)}(c^{-i})\|>ac^{-i(\frac{1}{\alpha_j}-\delta_j)}\right).
 	$$
 \end{lemma}
 
 \begin{proof}
 	For $N\in\nat$ and $n=1,\ldots,N$ define $Y_{n,N}:=X^{(j)}(k_n^N)-X^{(j)}(k_{n-1}^N)$, where $k_n^{N}:=\frac{n}{N}\,c^{-i}$. Then $Y_{1,N},\ldots,Y_{N,N}$ are independent and
 	$
 	\sum_{k=1}^n Y_{k,N} = X^{(j)}(k_n^N).
 	$
 	By Lemma \ref{LemmaA2A} for any $\varepsilon\in(0,1)$ there exists a constant $a_0>0$ such that for all $a\geq a_0$ we have
 	\begin{align*}
 	&\sup_{0\leq n\leq N}\Prob\left(\|\sum_{k=1}^N Y_{k,N} - \sum_{k=1}^n Y_{k,N}\|>ac^{-i(\frac{1}{\alpha_j}-\delta_j)}\right) \\
 	&\quad= \sup_{0\leq n\leq N}\Prob\left(\|X^{(j)}(k_N^N) - X^{(j)}(k_n^N)\|>ac^{-i(\frac{1}{\alpha_j}-\delta_j)}\right)\\
 	&\quad= \sup_{0\leq n\leq N}\Prob\left(\|X^{(j)}(k_N^N-k_n^N)\|>ac^{-i(\frac{1}{\alpha_j}-\delta_j)}\right)\\
 	&\quad= \sup_{0\leq n\leq N}\Prob\left(\|X^{(j)}\left(\left(1-\tfrac{n}{N}\right)c^{-i}\right)\|>ac^{-i(\frac{1}{\alpha_j}-\delta_j)}\right)\\
 	&\quad\leq\sup_{t\in[0,c^{-i}]}\Prob\left(\|X^{(j)}(t)\|>ac^{-i(\frac{1}{\alpha_j}-\delta_j)}\right)\leq \varepsilon <1.
 	\end{align*}
 	Using the L\'evy-Ottaviani inequality (see Lemma 3.21 in \cite{B}) and the fact that $(X^{(j)}(t))_{t\geq0}$ has right-continuous paths, it follows that
 	\begin{align*}
 	&\Prob\left(\sup_{t\in[0,c^{-i}]}\|X^{(j)}(t)\|>2ac^{-i(\frac{1}{\alpha_j}-\delta_j)}\right)\\
 	&\quad= \lim_{N\rightarrow\infty} \Prob\left(\sup_{0\leq n\leq N}\|X^{(j)}(k_n^N)\|>2ac^{-i(\frac{1}{\alpha_j}-\delta_j)}\right)\\
 	&\quad\leq \lim_{N\rightarrow\infty}\;\frac1{1-\varepsilon}\;\Prob\left(\|X^{(j)}(k_N^N)\|>ac^{-i(\frac{1}{\alpha_j}-\delta_j)}\right)\\
  	&\quad= \frac1{1-\varepsilon}\;\Prob\left(\|X^{(j)}(c^{-i})\|>ac^{-i(\frac{1}{\alpha_j}-\delta_j)}\right),
 	\end{align*}
 	concluding the proof.
 \end{proof}


For a L\'evy process $X=\{X(t):t\geq0\}$ define the first exit time from the closed ball $B(0,a)$
	$$
	P(a)=\inf\{t\geq0:\|X(t)\|>a\};
	$$
	and the maximum displacement process for $t>0$ as
	$$
	M(t)=\sup_{0\leq s\leq t}\|X(s)\|.
	$$
Note that for $a, r>0$ the first exit time $P(a)$ and the maximum displacement process $M(r)$ are related by
\begin{equation}
	\{P(a)<r\} = \{M(r)>a\}.
\end{equation}

 \begin{lemma}\label{UpperBoundPa}
Let $X$ be a $(c^E,c)$-operator semistable L\'evy process on $\rr^{d}$ with diagonal principal exponent $E_1$. Then for
 	\begin{align}\label{phi}
 	\phi(a)=\left\{\begin{array}{ll} a^{\alpha_1}\log\log\frac1a, & \text{if $X$ is of type $A$ and $0<\alpha_1<d_1$}  \\
 	\\
 	a^{\alpha_1}\left(\log\log\frac1a\right)^{1-\alpha_1}, & \text{if $X$ is of type $B$ and } 0<\alpha_1<1\end{array}\right.
 	\end{align}
 	there exist constants $K_{13},K_{14},\gamma_0>0$ such that
 	\begin{align}
 	\Prob\left(\sup_{\gamma\leq a\leq\delta}\frac{P(a)}{\phi(a)}<K_{13}\right) \leq \exp\left(-K_{14}\cdot\left(-\log\gamma\right)^{\frac18}\right),
 	\end{align}
 	for all $0<\gamma\leq\gamma_0$ and $\delta\geq\gamma^{1/6}$.
 \end{lemma}
 
 \begin{proof}
 	First assume that $X$ is of type $A$ and $\alpha_1\in(0,\min\{d_1,2\})$, thus $\phi(a)=a^{\alpha_1}\log\log\frac1a$. By regular variation techniques, it can be shown that for $\alpha_1<d_1$ the function $\psi$, defined by $\psi(s)=s^{1/\alpha_1}(\log\log 1/s)^{-1/\alpha_1}$, is asymptotically inverse to $\phi$ in the sense that
 	\begin{equation}\label{phipsiasym}
 	\phi(\psi(s))\sim s \text{ as } s\to0+ \quad \text{ and } \quad \psi(\phi(a))\sim a \text{ as } a\to0+.
 	\end{equation}
 	Owing to the fact that $\{M(t)>a\}=\{P(a)<t\}$, instead of estimating the probability that $\frac{P(a)}{\phi(a)}$ remains small, we will now estimate the probability that $\frac{M(t)}{\psi(t)}$ remains large. Therefore, define a sequence $t_k=\exp(-k^2)$, $k\geq1$. Then for all $k$ there exists an $i_k\in\nat_0$ such that $c^{-(i_k+1)}\leq t_k< c^{-i_k}$. Define 
 	\begin{equation*}
 	M'(t_k)=\sup_{t_{k+1}\leq t\leq t_k}\|X(t)-X(t_{k+1})\|
 	\end{equation*}
 	and $C_4:=(3K_{10})^{1/\alpha_1}$, where $K_{10}$ is the constant in Lemma \ref{lemmaA2} (a). Furthermore, define
 	\begin{equation*}
 	D_k:=\left\{\frac{M(t_k)}{\psi(t_k)}>2\cdot C_4\right\},\; G_k:= \left\{\frac{M'(t_k)}{\psi(t_k)}> C_4\right\} \text{ and } H_k:= \left\{\frac{M(t_{k+1})}{\psi(t_k)}> C_4\right\}.
 	\end{equation*}
 	Then
 	\begin{align*}
 	D_k&=\left\{\sup_{0\leq t\leq t_k}\|X(t)\|>2C_4\psi(t_k)\right\}\\
 	&=\left\{\sup_{0\leq t\leq t_{k+1}}\|X(t)\|>2 C_4\psi(t_k)\right\}\cup\left\{\sup_{t_{k+1}\leq t\leq t_k}\|X(t)\|>2C_4\psi(t_k)\right\}\\
 	&\subseteq \left\{\sup_{0\leq t\leq t_{k+1}}\|X(t)\|>2 C_4\psi(t_k)\right\}\\
 	&\quad\;\;\cup\left\{\sup_{t_{k+1}\leq t\leq t_k}\|X(t)-X(t_{k+1})\|+\|X(t_{k+1})\|>2 C_4\psi(t_k)\right\}\\
 	&\subseteq \left\{\sup_{0\leq t\leq t_{k+1}}\|X(t)\|>2 C_4\psi(t_k)\right\} \cup\left\{\|X(t_{k+1})\|> C_4\psi(t_k)\right\}\\
 	&\quad\;\;\cup\left\{\sup_{t_{k+1}\leq t\leq t_k}\|X(t)-X(t_{k+1})\|>C_4\psi(t_k)\right\}\\
 	&\subseteq \left\{\sup_{0\leq t\leq t_{k+1}}\|X(t)\|> C_4\psi(t_k)\right\} \cup \left\{\sup_{t_{k+1}\leq t\leq t_k}\|X(t)-X(t_{k+1})\|>C_4\psi(t_k)\right\}\\
 	&= H_k \cup G_k.
 	\end{align*}
 	And for all $m\in\nat$ this gives us
 	\begin{align*}
 	\bigcap_{k=m+1}^{2m} D_k \subseteq \left(\bigcap_{k=m+1}^{2m}G_k\right)\cup\left(\bigcup_{k=m+1}^{2m}H_k\right).
 	\end{align*}
 	Note that the sets $(G_k)_{k\in\nat}$ are pairwise independent. Set $\Prob(G_k)=1-p_k$ and $\Prob(H_k)=q_k$. Applying Lemma \ref{lemmaA2} (a) we have for sufficiently large $k$
 	\begin{align*}
 	p_k &= \Prob\left(M'(t_k)\leq C_4 \psi(t_k)\right)\\
 	& = \Prob\left(\sup_{t_{k+1}\leq t\leq t_k}\|X(t)-X(t_{k+1})\|\leq C_4\psi(t_k)\right)\\
 	& = \Prob\left(\sup_{0\leq t\leq t_k-t_{k+1}}\|X(t)\|\leq C_4\psi(t_k)\right)\\
 	&\geq\Prob\left(\sup_{0\leq t\leq t_k}\|X(t)\|\leq C_4 \cdot t_k^{\frac1{\alpha_1}}\left(\log\log\frac{1}{t_k}\right)^{-\frac1{\alpha_1}}\right)\\
 	&\geq \exp\left(-K_{10}\left(C_4\left(\log\log(1/t_k)\right)^{-\frac1{\alpha_1}}\right)^{-\alpha_1}\right)\\
 	&=\exp\left(-K_{10}\cdot \frac1{3K_{10}}\cdot \log\log\frac1{t_k}\right)\\
 	&=\exp\left(-\frac13\log\log\frac1{t_k}\right)=\exp\left(-\frac13\log k^2\right)=k^{-\frac23}.
 	\end{align*}
 	On the other hand, choosing $\delta'\in(0,\alpha_p)$, $\delta_1:=0$ and $\delta_j=\frac1{\alpha_j}-\frac1{\alpha_1}\in(0,\alpha_j^{-1})$ for $j=2,\ldots,p$, then for sufficiently large $k\in\nat$ we get by Lemma \ref{LemmaA2} and Lemma \ref{TailBehav} 
 	\begin{align*}
 	q_k &= \Prob\left(\sup_{0\leq t\leq t_{k+1}}\|X(t)\|>C_4\psi(t_k)\right)\\
 	&\leq \sum_{j=1}^p\Prob\left(\sup_{0\leq t\leq t_{k+1}}\|X^{(j)}(t)\|>\frac{C_4\cdot t_k^{\frac1{\alpha_1}}(\log\log\frac1{t_k})^{-\frac1{\alpha_1}}}{p}\right)\\
 	& = \sum_{j=1}^p\Prob\left(\sup_{0\leq t\leq t_{k+1}}\|X^{(j)}(t)\|>\frac{C_4\cdot t_k^{\frac1{\alpha_1}}(\log\log\frac1{t_k})^{-\frac1{\alpha_1}}t_{k+1}^{-(\frac1{\alpha_j}-\delta_j)}t_{k+1}^{\frac1{\alpha_j}-\delta_j}}{p}\right)\\
	& \leq \sum_{j=1}^p \Prob\left(\sup_{0\leq t\leq c^{-i_{k+1}}}\|X^{(j)}(t)\|>\frac{C_4\cdot t_k^{\frac1{\alpha_1}}\left(\log\log \frac1{t_k}\right)^{-\frac1{\alpha_1}}t_{k+1}^{-(\frac1{\alpha_j}-\delta_j)}\cdot c^{-(i_{k+1}+1)(\frac1{\alpha_j}-\delta_j)}}{p}\right)\\
 	&\leq \sum_{j=1}^p K^{(j)}\cdot\Prob\left(\|X^{(j)}(c^{-i_{k+1}})\|>\frac{C_4\cdot t_k^{\frac1{\alpha_1}}\left(\log\log \frac1{t_k}\right)^{-\frac1{\alpha_1}}t_{k+1}^{-\frac1{\alpha_1}}\cdot c^{-\frac1{\alpha_1}}\cdot c^{-i_{k+1}(\frac1{\alpha_j}-\delta_j)}}{2p}\right)\\
 	&\leq \sum_{j=1}^p \widetilde K^{(j)}\cdot \left(t_k^{\frac1{\alpha_1}}\left(\log\log \tfrac1{t_k}\right)^{-\frac1{\alpha_1}}t_{k+1}^{-\frac1{\alpha_1}}\right)^{-(\alpha_j-\delta')}\\
 	& \leq \sum_{j=1}^p \widetilde{K}^{(j)}\left(t_k^{-1}t_{k+1}\log\log\tfrac1{t_k}\right)^{\frac{\alpha_j-\delta'}{\alpha_1}}\leq K\cdot \left(t_k^{-1}t_{k+1}\log\log\tfrac1{t_k}\right)^{\frac{\alpha_p-\delta'}{\alpha_1}}\\
 	& = K\cdot\left(\exp\left(k^2-(k+1)^2\right)\cdot\log\left(k^2\right)\right)^{\frac{\alpha_p-\delta'}{\alpha_1}}= K\cdot\left(\exp(-2k-1)\cdot\log\left(k^2\right)\right)^{\frac{\alpha_p-\delta'}{\alpha_1}}\\
 	&\leq\exp(-C_{41}\,k).
 	\end{align*}
 	Hence, there exists $m_0\in\nat$ large enough such that for  all $m>m_0$
 	\begin{align*}
 	&\Prob\left(\bigcap_{k=m+1}^{2m}D_k\right)\leq\Prob\left(\bigcap_{k=m+1}^{2m}G_k\right)+\sum_{k=m+1}^{2m}\Prob(H_k)\\
 	&= \prod_{k=m+1}^{2m}(1-p_k) + \sum_{k=m+1}^{2m}q_k\leq \prod_{k=m+1}^{2m}\exp(-p_k) + \sum_{k=m+1}^{2m}\exp(-C_{41}\,k)\\
 	&\leq \exp\left(-\sum_{k=m+1}^{2m}p_k\right) + C_5\exp(-C_{41}\,m)\\
 	&\leq \exp\left(-\sum_{k=m+1}^{2m}k^{-\frac23}\right) + C_5\exp(-C_{41}\,m)\\
 	&\leq \exp\left(-\sum_{k=m+1}^{2m}(2m)^{-\frac23}\right) + C_5\exp(-C_{41}\,m)\\
 	&= \exp\left(-2^{-\frac23}m^{\frac13}\right) + C_5\exp(-C_{41}\,m)\leq \exp\left(-m^{\frac14}\right).
 	\end{align*}
 	Define
 	\begin{equation*}
 	a_k = 2C_4\psi(t_k) = 2C_4\left(\log k^2\right)^{-\frac1{\alpha_1}}\cdot\exp\left(-{k^2}/{\alpha_1}\right)
 	\end{equation*}
 	Then for $m_0$ sufficiently large one can show that $a_{2m}>a_m^4$, for all $m\geq m_0$. By \eqref{phipsiasym} and the properties of $\phi$, there now exists a positive constant $K_{13}>0$ such that for $k$ large enough
 	\begin{equation*}
 	t_k\sim\phi\left(\frac{a_k}{2C_4}\right)\geq 2 K_{13}\cdot\phi(a_k).
 	\end{equation*}
 	Therefore, using $\{M(t)>a\}=\{P(a)<t\}$, we have for $m\geq m_0$ sufficiently large,
 	\begin{align*}
 	&\Prob\left(\bigcap_{k=m+1}^{2m}D_k\right) = \Prob\left(\bigcap_{k=m+1}^{2m}\left\{M(t_k)>2C_4\psi(t_k)\right\}\right)\\
 	&= \Prob\left(\bigcap_{k=m+1}^{2m}\left\{P(2C_4\psi(t_k))<t_k\right\}\right) \geq \Prob\left(\bigcap_{k=m+1}^{2m}\left\{P(a_k)<K_{13}\phi(a_k)\right\}\right)\\
 	&\geq \Prob\left(\sup_{a_{2m}\leq a\leq a_{m+1}}\frac{P(a)}{\phi(a)}<K_{13}\right).
 	\end{align*}
 	Let $\gamma_0>0$ be small enough such that $\gamma_0\leq\exp(-5m_0^2/\alpha_1)$. Let $0<\gamma<\gamma_0$, $\delta>\gamma^{1/6}$ and $m$ be the largest integer less than $\sqrt{-\frac{\alpha_1}{5}\log\gamma}$. Then
 	\begin{align*}
 	\gamma&\leq \exp(-5m^2/\alpha_1)\leq a_m^4 < a_{2m} < a_{m+1}\\
 	&=2C_4\left(\log\left((m+1)^2\right)\right)^{-1/\alpha_1}\cdot \exp\left(-\frac{(m+1)^2}{\alpha_1}\right)\\
 	&\leq  2 C_4 \left(\log\left((m+1)^2\right)\right)^{-1/\alpha_1}\cdot \exp\left(\frac{\alpha_1\log\gamma}{5\alpha_1}\right)\\
 	& =  2 C_4 \left(\log\left((m+1)^2\right)\right)^{-1/\alpha_1}\cdot\gamma^{1/5}\\
 	& \leq 2 C_4 \left(\log\log\left(\gamma^{-\alpha_1/5}\right)\right)^{-1/\alpha_1}\cdot\gamma^{1/5} \leq\gamma^{1/6}<\delta,
 	\end{align*}
 	and hence
 	\begin{align*}
 	&\Prob\left(\sup_{\gamma\leq a\leq \delta}\frac{P(a)}{\phi(a)}<K_{13}\right)\leq \Prob\left(\bigcap_{k=m+1}^{2m}D_k\right)\leq\exp\left(-m^\frac14\right)\\
 	& \leq \exp\left(-K\cdot(m+1)^\frac14\right) \leq \exp\left(-K\cdot\left(-\frac{\alpha_1}{5}\log\gamma\right)^\frac18\right)\\
 	& \leq \exp\left(-K_{14}\cdot(-\log\gamma)^\frac18\right),
 	\end{align*}
 	which concludes the proof for $X$ of type $A$ and $0<\alpha_1<2$.
 	
 	Now assume that $X$ is of type $B$ and $0<\alpha_1<1$. Then $\phi(a)=a^{\alpha_1}\left(\log\log\frac1a\right)^{1-\alpha_1}$. Define $\psi(s)=s^{1/\alpha_1}\left(\log\log\frac1s\right)^{-\frac{1-\alpha_1}{\alpha_1}}$. Then $\phi$ and $\psi$ are asymptotically inverse to each other as $a, s\to 0$ in the same sense as in \eqref{phipsiasym}. Again consider the sequence $t_k=\exp(-k^2)$, $k\geq1$. For all $k$ there exists an $i_k\in\nat$ such that $c^{-(i_k+1)}\leq t_k < c^{-i_k}$. Furthermore, define $C_{6}:=(3K_{11})^{(1-\alpha_1)/\alpha_1}$, where $K_{11}$ is as in Lemma \ref{lemmaA2} (b), and let
 	\begin{equation*}
 	D_k:=\left\{\frac{M(t_k)}{\psi(t_k)}>2\cdot C_{6}\right\},\; G_k:= \left\{\frac{M'(t_k)}{\psi(t_k)}> C_{6}\right\} \text{ and } H_k:= \left\{\frac{M(t_{k+1})}{\psi(t_k)}> C_{6}\right\}.
 	\end{equation*}
 	With the same methods as before one can show that for all $m\in\nat$
 	\begin{align*}
 	\bigcap_{k=m+1}^{2m} D_k \subseteq \left(\bigcap_{k=m+1}^{2m}G_k\right)\cup\left(\bigcup_{k=m+1}^{2m}H_k\right).
 	\end{align*}
 	Set $\Prob(G_k)= 1-p_k$ and $\Prob(H_k)=q_k$. Applying Lemma \ref{lemmaA2}(ii) we have for sufficiently large $k$
 	\begin{align*}
 	p_k&\geq\Prob\left(\sup_{t_{k+1}\leq t\leq t_k}\|X(t)-X(t_{k+1})\|>C_{6}\psi(t_k)\right)\\
 	&\geq\Prob\left(\sup_{0\leq t\leq t_k}\|X(t)\|\leq C_{6}\cdot t_k^{1/\alpha_1}\left(\log\log\frac1{t_k}\right)^{-\frac{1-\alpha_1}{\alpha_1}}\right)\\
 	&\geq \exp\left(-K_{11}\left(C_{6}\left(\log\log (1/{t_k})\right)^{-\frac{1-\alpha_1}{\alpha_1}}\right)^{-\frac{\alpha_1}{1-\alpha_1}}\right)\\
 	& = \exp\left(-\frac{K_{11}}{3K_{11}}\log\log\frac1{t_k}\right) = k^{-\frac23}.
 	\end{align*}
On the other hand, choosing $\delta'\in(0,\alpha_p)$, $\delta_1:=0$ and $\delta_j=\frac1{\alpha_j}-\frac1{\alpha_1}\in(0,\alpha_j^{-1})$ for $j=2,\ldots,p$, similarly to type A for sufficiently large $k\in\nat$ we get by Lemma \ref{LemmaA2} and Lemma \ref{TailBehav} 
 	\begin{align*}
 	q_k &\leq \sum_{j=1}^p\Prob\left(\sup_{0\leq t\leq t_{k+1}}\|X^{(j)}(t)\|>\frac{C_{6}\cdot t_k^{\frac1{\alpha_1}}(\log\log\frac1{t_k})^{-\frac{1-\alpha_1}{\alpha_1}}}{p}\right)\\
         & = \sum_{j=1}^p\Prob\left(\sup_{0\leq t\leq t_{k+1}}\|X^{(j)}(t)\|>\frac{C_6\cdot t_k^{\frac1{\alpha_1}}(\log\log\frac1{t_k})^{-\frac{1-\alpha_1}{\alpha_1}} t_{k+1}^{-(\frac1{\alpha_j}-\delta_j)}t_{k+1}^{\frac1{\alpha_j}-\delta_j}}{p}\right)\\
	& \leq \sum_{j=1}^p \Prob\left(\sup_{0\leq t\leq c^{-i_{k+1}}}\|X^{(j)}(t)\|>\frac{C_6\cdot t_k^{\frac1{\alpha_1}}\left(\log\log \frac1{t_k}\right)^{-\frac{1-\alpha_1}{\alpha_1}} t_{k+1}^{-(\frac1{\alpha_j}-\delta_j)}\cdot c^{-(i_{k+1}+1)(\frac1{\alpha_j}-\delta_j)}}{p}\right)\\
 	&\leq \sum_{j=1}^p K^{(j)}\cdot\Prob\left(\|X^{(j)}(c^{-i_{k+1}})\|>\frac{C_6\cdot t_k^{\frac1{\alpha_1}}\left(\log\log \frac1{t_k}\right)^{-\frac{1-\alpha_1}{\alpha_1}} t_{k+1}^{-\frac1{\alpha_1}}\cdot c^{-\frac1{\alpha_1}}\cdot c^{-i_{k+1}(\frac1{\alpha_j}-\delta_j)}}{2p}\right)\\
 	&\leq \sum_{j=1}^p \widetilde K^{(j)}\cdot \left(t_k^{\frac1{\alpha_1}}\left(\log\log \tfrac1{t_k}\right)^{-\frac{1-\alpha_1}{\alpha_1}}t_{k+1}^{-\frac1{\alpha_1}}\right)^{-(\alpha_j-\delta')}\\
 	& \leq \sum_{j=1}^p \widetilde{K}^{(j)} \left(t_k^{-1}t_{k+1}\left(\log\log\tfrac1{t_k}\right)^{1-\alpha_1}\right)^{\frac{\alpha_j-\delta'}{\alpha_1}}\leq K\cdot \left(t_k^{-1}t_{k+1}\left(\log\log\tfrac1{t_k}\right)^{1-\alpha_1}\right)^{\frac{\alpha_p-\delta'}{\alpha_1}}\\
 	& = K\cdot\left(\exp\left(k^2-(k+1)^2\right)\cdot\left(\log\left(k^2\right)\right)^{1-\alpha_1}\right)^{\frac{\alpha_p-\delta'}{\alpha_1}}\\
	& = K\cdot\left(\exp(-2k-1)\cdot\left(\log\left(k^2\right)\right)^{1-\alpha_1}\right)^{\frac{\alpha_p-\delta'}{\alpha_1}}\leq \exp(-C_{61}\,k).
 	\end{align*}
Analogously to the calculations in type A above, we can now prove that there exist constants $K_{13}, K_{14}, \gamma_0>0$ such that
 	\begin{align}
 	\Prob\left(\sup_{\gamma\leq a\leq\delta}\frac{P(a)}{\phi(a)}<K_{13}\right) \leq \exp\left(-K_{14}\cdot\left(-\log\gamma\right)^{\frac18}\right),
 	\end{align}
 	provided $0<\gamma\leq\gamma_0$ and $\delta\geq\gamma^{1/6}$. This concludes the proof.
 \end{proof}
 
 
 \begin{cor}\label{UpperBoundTa1}
 	Let $X$ be a $(c^E,c)$-operator semistable L\'evy process on $\rd$ with diagonal principal exponent $E_1$ and $\alpha_1<d_1$. Then for $\phi$ as in \eqref{phi} there exist constants $K_{13}$, $K_{14}, \gamma_0 > 0$ such that
 	\begin{align}
 	\Prob\left(\sup_{\gamma\leq a\leq\delta}\frac{T(a,1)}{\phi(a)}<K_{13}\right) \leq \exp\left(-K_{14}\cdot\left(-\log\gamma\right)^{\frac18}\right),
 	\end{align}
 	for all $0<\gamma\leq\gamma_0$ and $\delta\geq\gamma^{1/6}$.
 \end{cor}
 
 \begin{proof}
 	Obviously, $T(a,1)\leq t<1$ implies that $P(a)\leq t$. This gives us
 	\begin{align*}
 	\left\{\sup_{\gamma\leq a\leq\delta}\frac{T(a,1)}{\phi(a)}<K_{13}\right\}\subseteq  	\left\{\sup_{\gamma\leq a\leq\delta}\frac{P(a)}{\phi(a)}<K_{13}\right\},
 	\end{align*}
 	provided $\delta$ and therefore $\gamma$ small enough to ensure that $\phi$ is increasing on $(0,\delta)$ and $K_{13}\phi(\delta)<1$. Lemma \ref{UpperBoundPa} then concludes the proof.
 \end{proof}
 
 Let $K_3>0$ be a fixed constant. A family $\Lambda(a)$ of cubes of side $a$ in $\rd$ is called $K_3$-nested if no balls of radius $a$ in $\rd$ can intersect more than $K_3$ cubes of $\Lambda(a)$. Here, we will choose $\Lambda(a)$ to be the family of all cubes in $\rd$ of the form $[k_1a,(k_1+1)a]\times\ldots\times[k_da,(k_d+1)a]$ with $K_3=3^d$. The following covering lemma is due to Pruitt and Taylor \cite[Lemma 6.1]{PT}
 
 \begin{lemma}\label{PT}
 	Let $X=\{X(t)\}_{t\geq0}$ be a  L\'evy process in $\rd$ and let $\Lambda(a)$ be a fixed $K_{3}$-nested family of cubes in $\rd$ of side $a$ with $0<a\leq1$. For any $u\geq0$ let $M_{u}(a,s)$ be the number of cubes in $\Lambda(a)$ hit by $X(t)$ at some time $t\in [u,u+s]$. Then
 	$$\Exp\left[M_{u}(a,s)\right]\leq 2\,K_{3}s\cdot\left(\Exp\left[T\left(\tfrac{a}{3},s\right)\right]\right)^{-1}.$$
 \end{lemma}
 
 For $u=0$ we simply write $M(a,s):=M_0(a,s)$. The following result is a direct consequence of  Lemma \ref{sojourndiag} and Lemma \ref{PT}. Although part (ii) is not needed here, it might be useful to show that $\phi-m(X([0,1]))<\infty$ for $\phi$ as in \eqref{phitypb} in case $X$ is of type $A$, $d\geq2$ and $\alpha_1>d_1=1$.
 
 \begin{lemma}\label{NumberCubes}
 	Let $X$ be a $(c^E,c)$-operator semistable L\'evy process on $\rd$ with diagonal principal exponent $E_1$.
 	\begin{itemize}
 		\item[(i)] If $\alpha_1<d_1$, there exists a constant $K_{15}>0$ such that for all $a\leq1$
 		\begin{equation*}
 		\Exp[M(a,1)]\leq K_{15}a^{-\alpha_1}.
 		\end{equation*}
 		\item[(ii)] If $d\geq2$ and $\alpha_1>d_1$, then $d_1=1$ and we further assume that $E_2$ is diagonal. Then there exists a constant $K_{16}>0$ such that for all $a>0$ small enough
 		\begin{equation*}
 		\Exp[M(a,1)]\leq K_{16}a^{-\rho},
 		\end{equation*}
 		where $\rho=1+\alpha_2(1-1/\alpha_1)$.
 	\end{itemize}
 \end{lemma}
 
Let $\Lambda_k$ be the set of cubes of side $2^{1-k}$ and centered at $(j_1/2^k,\ldots,j_d/2^k)$, where $j_l$, $1\leq l\leq d$, are integers, closed on the left and open on the right. The following result is taken from Lemma 3.9 in \cite{HY} and based on Lemma 9 in \cite{T}.
 
 \begin{lemma}\label{Cubes}
 	If $E=\bigcup_{i=1}^m I_i$, where each $I_i$ is a cube of $\Lambda_k$ for some integer $k$, then we can find a subset $\{j_r\}$ such that $E\subseteq\bigcup I_{j_r}$ and no point of $E$ is contained in more than $2^d$ of the cubes $I_{j_r}$.
 \end{lemma}
 
 \begin{theorem}\label{lessthaninfinity}
 	Let $X$ be a $(c^E,c)$-operator semistable L\'evy process on $\rd$ with diagonal principal exponent $E_1$ and $\alpha_1<d_1$. Then for 
 	\begin{align*}
 	\phi(a)=\left\{\begin{array}{ll} a^{\alpha_1}\log\log\frac1a, & \text{if $X$ is of type $A$ and } 0<\alpha_1<2  \\
 	\\
 	a^{\alpha_1}\left(\log\log\frac1a\right)^{1-\alpha_1}, & \text{if $X$ is of type $B$ and } 0<\alpha_1<1\end{array}\right.
 	\end{align*}
 	we have almost surely $\phi-m(X([0,1]))<\infty$.
 \end{theorem}
 
 \begin{proof}
 	Let $r$ be a positive integer and $\delta:=2^{-r}$. Furthermore, Let $n$ be an integer with $2^{-n}\leq\min(\gamma_0,2^{-6r})$, where $\gamma_0$ is as in Lemma \ref{UpperBoundPa} and $\overline{\Lambda}_n$ the collection of cubes of side $2^{-n}$  with centers the same as in $\Lambda_n$. Define $\tau^I=\inf\;\{t\geq0:X(t)\in I\}$ for any cube $I$ and $\overline{\Lambda'}_n = \{I\in\overline{\Lambda}_n:\tau^I\leq 1\}$, the cubes hit by $X$ over the time interval $[0,1]$. Then $M(2^{-n},1)=|\overline{\Lambda'}_n|$. Let $\gamma_n:=2^{-n}$. We say that a cube $I$ in $\overline{\Lambda'}_n$ is bad if for all $a\in[\gamma_n,\delta]$
 	\begin{equation*}
 	\int_{\tau^I}^{\tau^I+1}1_{B(X(\tau^I),a)}(X(t))dt \leq K_{13}\phi(a),
 	\end{equation*}
 	and good otherwise. For any cube $I\in\overline{\Lambda}_n$ we have
 	\begin{align*}
 	&\Prob\left(I \text{ is bad }| 0\leq\tau^I\leq1\right)\\
 	&= \Prob\left(\sup_{\gamma_n\leq a\leq\delta}\left.\left\{\frac{\int_{\tau^I}^{\tau^I+1}1_{B(X(\tau^I),a)}(X(t))dt}{\phi(a)}\right\}\leq K_{13}\right|0\leq\tau^I\leq1\right)\\
 	&= \Prob\left(\sup_{\gamma_n\leq a\leq\delta}\left.\left\{\frac{\int_{0}^{1}1_{B(0,a)}(X(t+\tau^I)-X(\tau^I))dt}{\phi(a)}\right\}\leq K_{13}\right|0\leq\tau^I\leq1\right).
 	\end{align*}
 	Note that $\{X(t + \tau^I)-X(\tau^I)\}_{t\geq0}$ is identical in law with $\{X(t)\}_{t\geq0}$ on $\{\tau^I<\infty\}$ by the strong Markov property (see e.g. Corollary 40.11 in \cite{Sato}). Hence, we get by applying Corollary \ref{UpperBoundTa1}
 	\begin{align*}
 	\Prob\left(I \text{ is bad }| 0\leq\tau^I\leq1\right)\leq \exp\left(-K_{14}\cdot\left(-\log\gamma_n\right)^{\frac18}\right) = \exp\left(-C_7\cdot n^{\frac18}\right),
 	\end{align*}
 	where $C_7>0$ is a constant independent from $n$. Now let $N_n$ denote the number of bad cubes in $\overline{\Lambda'}_n$. Then by Lemma \ref{NumberCubes}
 	\begin{align*}
 	\Exp[N_n]\leq \exp\left(-C_7\cdot n^{\frac18}\right)\Exp\left[M(2^{-n},1)\right]\leq K_{15} 2^{n\alpha_1}\exp\left(-C_7\cdot n^{\frac18}\right).
 	\end{align*}
 	Hence, by the Markov inequality for $n$ sufficiently large there exists a constant $C_8>0$
 	\begin{align*}
 	&\Prob\left(N_n\geq2^{n\alpha_1}\exp\left(-n^{\frac1{10}}\right)\right) \leq \frac{\Exp\left[|N_n|\right]}{2^{n\alpha_1}\exp\left(-n^{1/10}\right)}\\
 	&\leq K_{15}\exp\left(-C_7\cdot n^{\frac18}+n^{\frac1{10}}\right)\leq K_{15}\exp\left(-C_8\cdot n^{\frac1{10}}\right).
 	\end{align*}
 	This implies that
 	\begin{align*}
 	\sum_{n=1}^{\infty}\Prob\left(N_n\geq2^{n\alpha_1}\exp\left(-n^{1/10}\right)\right)\leq K + K_{15}	\sum_{n=1}^{\infty}\exp\left(-C_8\cdot n^{1/10}\right)<\infty.
 	\end{align*}
 	By the Borel-Cantelli lemma, there now exists an $\Omega_0$ with $\Prob(\Omega_0)=1$ such that for all $\omega\in\Omega_0$ we can find an integer $n_1=n_1(\omega)$ such that for $n\geq n_1$
 	\begin{align*}
 	N_n(\omega)< 2^{n\alpha_1}\exp\left(-n^{1/10}\right).
 	\end{align*}
 	Furthermore, by regular variation techniques, there exists a constant $C_9>0$ such that
 	\begin{align*}
 	\phi\left(\sqrt{d}\cdot 2^{-n}\right) = \left(\sqrt{d}\cdot 2^{-n}\right)^{\alpha_1}\log\log\left(\sqrt{d}\cdot 2^n\right) = C_9\; 2^{-n\alpha_1}\log (n).
 	\end{align*}
 	Thus, for $\omega\in\Omega_0$ and $n\geq n_1(\omega)$
 	\begin{equation}\label{eqnBad}
 	\sum_{I\text{ bad }}\phi(|I|) = N_n(\omega)\cdot \phi\left(\sqrt{d}\cdot 2^{-n}\right)\leq C_9\exp\left(-n^{1/10}\right)\log(n)
 	\end{equation}
 	
 	Now consider the good cubes $I$ in $\overline{\Lambda'}_n$. Our aim is to show that the good cubes can be covered economically. For a good cube $I$ there exists $a\in[\gamma_n,2^{-r}]$ such that
 	\begin{align*}
 	\phi(a)<\frac1{K_{13}}\int_{\tau^I}^{\tau^I+1}1_{B(X(\tau^I),a)}(X(t))dt.
 	\end{align*}
 	We can find an integer $k$ with $2^{-k}>5a\geq 2^{-k-1}$ and a cube $I'$ in $\Lambda_k$ such that $I'$ contains $I$ and $B(X(\tau^I),a)$. Then, one can easily show that $k>r-4$ and, since $\tau^I\leq1$ by definition of $\overline{\Lambda'}_n$, we get
 	\begin{align*}
 	&\phi\left(|I|\right) = \phi\left(\sqrt{d}\cdot 2^{-k+1}\right) = \phi\left(\sqrt{d}\cdot 4\cdot 2^{-k-1}\right)\leq \phi\left(\sqrt{d}\cdot 4\cdot 5a\right)\\
 	&\leq K\phi(a) < K\int_{\tau^I}^{\tau^I+1}1_{B(X(\tau^I),a)}(X(t))dt\leq  K \int_0^2 1_{I'}(X(t))dt.
 	\end{align*} 
 	Applying Lemma \ref{Cubes} to the collection $\{I':I \text{ is good}\}$, we can show that there is a subset $\Lambda'$, which still covers $\bigcup_{I \text{ good }}I$, but no point is covered more than $2^d$ times. Hence, $\sum_{I'\in\Lambda'}1_{I'}\leq2^d$ and there exists a constant $C_{10}>0$ such that
 	\begin{align}\label{eqnGood}
 	\sum_{I'\in\Lambda'}\phi(|I'|)\leq\sum_{I'\in\Lambda'} K \int_0^2 1_{I'}(X(t))dt = K \int_0^2 \sum_{I'\in\Lambda'} 1_{I'}(X(t))dt \leq C_{10} \cdot 2^{d+1}.
 	\end{align}
 	Using all the bad cubes together with the covering of good cubes defined above, we obtain a covering of $X([0,1])$ by cubes with diameters less than $\sqrt{d}\cdot 2^{-r+5}$. This means
 	\begin{align*}
 	X([0,1])\subseteq\left(\bigcup_{I \text{ bad}}I\right)\cup\left(\bigcup_{I'\in\Lambda'} I'\right).
 	\end{align*}
 	For sufficiently large $n$ applying \eqref{eqnBad} and \eqref{eqnGood} we finally arrive at
 	\begin{align*}
 	\sum_{I: \text{ bad }} \phi\left(|I|\right) + \sum_{I'\in\Lambda'}\phi(|I'|) \leq C_9\exp\left(-n^{1/10}\right)\log(n) + C_{10} \cdot 2^{d+1}\leq C_{10} \cdot 2^{d+1} + 1.
 	\end{align*}
 	Thus, $\phi-m(X([0,1]))\leq C_{10} \cdot 2^{d+1} + 1 <\infty$ almost surely, which concludes the proof.
 \end{proof}


\subsection{Proof of the main result}

The proof of Theorem \ref{mainresult} now follows directly from Theorem \ref{greaterthanzero} and Theorem \ref{lessthaninfinity}.

\bibliographystyle{plain}

\end{document}